\documentclass[11pt]{article}

\usepackage{amssymb}
\usepackage{mathtools}
\usepackage[numbers]{natbib}
\usepackage[margin=1in]{geometry}

\usepackage{titlesec}
\titleformat{\section}
{\centering\large\scshape}{\thesection.}{.5em}{}
\titleformat{\subsection}
{\large}{\thesubsection.}{.5em}{\bfseries}

\usepackage{amsthm}
\newtheorem{thm}{Theorem}[section]
\newtheorem{lem}[thm]{Lemma}
\newtheorem{prop}[thm]{Proposition}

\theoremstyle{definition}

\theoremstyle{remark}

\usepackage{stucky}

\begin{document}
\bibliographystyle{abbrv}

\begin{center}
	{\large\bf THE SIXTH MOMENT OF AUTOMORPHIC $L$-FUNCTIONS}\\[1.5em]
	{\scshape Joshua Stucky}
\end{center}
\vspace{.5em}

\begin{abstract}
We investigate the sixth moment of the family of $L$-functions associated to holomorphic modular forms on $GL_2$ with respect to a congruence subgroup $\Gamma_1(q)$. We improve on previous work and obtain an unconditional upper bound of the correct order of magnitude.
\end{abstract}

\section{Introduction}

Moments of $L$-functions are among the central objects of study in modern analytic number theory, and there is a vast literature on the subject. In this paper, we shall be concerned with a family of $L$-functions attached to automorphic forms on $GL_2$. Specifically, we consider the sixth moment of $L$-functions associated to the family of holomorphic modular forms with respect to the congruence subgroup $\Gamma_1(q)$ (see \citep{Iwaniec} for definitions). Our work is motivated by the work of Djankovi\'{c} \citep{Djankovic} and Chandee and Li \citep{CL} on this family. For a detailed introduction to this family of $L$-functions, see the introductions of the above two papers.

Let $S_k(\Gamma_1(q))$ denote the space of holomorphic cusp forms on $\Gamma_1(q)$. We assume $k\geq 3$ is an odd integer and $q$ is prime (these assumptions are made mostly to eliminate oldforms). Then $S_k(\Gamma_1(q))$ is a Hilbert space with the Petersson’s inner product
\[
\chev{f,g} = \int_{\substack{\Gamma_1(q)\backslash \mathbb{H}\\ \chi(-1)=(-1)^k}} f(z) \bar{g}(z) y^{k-2} dx\ dy,
\]
and
\[
S_k(\Gamma_1(q)) = \bigoplus_{\chi \mod{q}} S_k(\Gamma_0(q),\chi).
\]
Let $\colh_\chi$ be an orthogonal basis for $S_k(\Gamma_0(q),\chi)$ consisting of Hecke cusp forms, normalized so that the first Fourier coefficient is 1. For each $f \in \colh_\chi$, we let $L(f,s)$ be the $L$-function associated to $f$, defined for $\Re s > 1$ as
\[
L(f,s) = \sum_{n\geq 1} \frac{\lambda_f(n)}{n^s} = \prod_p \pth{1-\frac{\lambda_f(p)}{p^s}+\frac{\chi(p)}{p^{2s}}}^{-1},
\]
where $\set{\lambda_f(n)}$ are the Hecke eigenvalues of $f$. In general, these satisfy the Hecke relation
\begin{equation}\label{eq:HeckeRelation}
	\lambda_f(m)\lambda_f(n) = \sum_{d\mid(m,n)} \chi(d)\lambda_f\fracp{mn}{d^2}.
\end{equation}
We define the completed $L$-function as
\begin{equation}\label{eq:CompletedLFunction}
	\Lambda\pth{f,\half+s} = \fracp{q}{4\pi^2}^{\frac{s}{2}}\Gamma\pth{s+\frac{k}{2}}L\pth{f,\half+s}.
\end{equation}
This satisfies the functional equation
\[
\Lambda\pth{f,\half+s} = i^k \bar{\eta_f} \Lambda\pth{\bar{f},\half-s}
\]
where $\abs{\eta_f}=1$ when $f$ is a newform. We define the harmonic average over $\colh_\chi$ by
\[
\sumh_{f\in\colh_\chi} \alpha_f = \frac{\Gamma(k-1)}{(4\pi)^{k-1}} \sum_{f\in\colh_\chi} \frac{\alpha_f}{\norm{f}^2},
\]
where $\norm{f}$ represents the norm given by Petersson's inner product. We are interested in the sixth moment
\[
\colm(q) = \frac{2}{\phi(q)} \sum_{\substack{\chi \mod{q}\\ \chi(-1)=(-1)^k}} \sumh_{f\in\colh_\chi} \abs{L\pth{f,\half}}^6.
\]
The conjectured asymptotic (see \citep{CL}) is
\[
\colm(q) = \frac{2}{\phi(q)} \sum_{\substack{\chi \mod{q}\\ \chi(-1)=(-1)^k}} \sumh_{f\in\colh_\chi} \abs{L\pth{f,\half}}^6 \sim \colc_q (\log q)^9
\]
where $\colc_q \asymp 1$ is an explicit constant (see (1.4) of \cite{CL} for a precise description). Using the asymptotic large sieve for the Fourier coefficients of cusp forms developed by Iwaniec and Xiaoqing Li \citep{OrthoHecke}, Djankovi\'{c} \citep{Djankovic} has shown 
\[
\colm(q) \ll q^\ep
\]
for any $\ep > 0$, whereas Chandee and Xiannan Li \citep{CL} have obtained the following asymptotic formula for the smoothed sixth moment:
\[
\begin{aligned}
	\frac{2}{\phi(q)} \sum_{\substack{\chi \mod{q}\\ \chi(-1)=(-1)^k}} \sumh_{f\in\colh_\chi} \int_{-\infty}^{\infty} \abs{\Lambda\pth{f,\fract{1}{2}+it}}^6 dt  \sim \colc_q (\log q)^9 \int_{-\infty}^{\infty} \abs{\Gamma(\fract{k}{2}+it)}^6 dt.
\end{aligned}
\]
Note that the integral in $t$ is quite short due to the presence of the gamma function. Building on these results, we prove

\begin{thm}\label{thm:Main}
	Let $q$ be prime and $k\geq 3$. Then, as $q\to\infty$, we have
	\[
	\colm(q) \ll (\log q)^9.
	\]
\end{thm}

Our proof of Theorem \ref{thm:Main} adheres closely to the work of Chandee and Li \citep{CL} and may be seen as an application of their ideas that avoids many of the technical details of their proof. Although we sacrifice an asymptotic, our result has the benefit of having no integral in $t$ and being of the correct order of magnitude.

\subsection{Notation}
We use a bold letter such as $\bolda$ to denote the pair of variables $a_1,a_2$ and write $f(\bolda)$ to indicate that $f$ is a function depending on these variables. However, we use $\boldn$ and $\boldN$ to indicate the pairs  $(n,m)$ and $(N,M)$, respectively. We write $n\asymp N$ to denote the condition $c_1 N\leq n\leq c_2 N$ for some suitable constants $0 < c_1 < c_2$. The use of the notation $\sumstar$ in a sum such as $\sumstar_{x(c)}$ indicates that the is sum over residue classes $x$ which are coprime to the modulus of the sum, in this case $c$. In such a sum, we denote by $\bar{x}$ the inverse of $x$ mod $c$. All other notation should be clear from context.

\subsection{Acknowledgments} This work was accomplished in partial fulfillment of my PhD. As such, I would like to thank my advisor, Xiannan Li, for suggesting this problem to me and for many helpful comments in the development of these results. As well, I thank the anonymous referee for close readings of the initial submission and several revisions of this paper. Their suggestions have greatly improved the readability of this paper. 

\section{Outline of the Paper}\label{sec:Outline}

To help orient the reader, we provide an outline for the proof. First, after applying the approximate functional equation for $L\pth{f,\half}^3$, the main object we need to understand is roughly of the form
\[
\frac{2}{\phi(q)} \sum_{\substack{\chi \mod{q}\\ \chi(-1)=(-1)^k}} \sumh_{f\in\colh_\chi} \sum_{m,n\asymp q^{\frac{3}{2}}} \frac{\tau_3(m)\tau_3(n)\lambda_f(m)\lambda_f(n)}{\sqrt{mn}}.
\]
We apply the functional equation for $L\pth{f,\half}^3$ rather than for $L\pth{f,\half}^6$ to avoid unbalanced sums in $m$ and $n$ (i.e. $m,n \asymp q^{3/2}$ rather than the weaker condition $mn \leq q^3$). We note that the $t$ integral used \cite{CL} is included for precisely the same reason. It is also worth noting that the application of Cauchy's inequality in (\ref{eq:FirstCauchy}) immediately precludes any hope of obtaining an asymptotic formula by our method, as we completely ignore the arithmetic of the root numbers $\eta_f$. Applying Petersson’s formula to the average over $f\in\colh_\chi$ leads to diagonal terms $m = n$ and off-diagonal terms. The diagonal terms are evaluated fairly easily in Section \ref{sec:Diagonal}. The off-diagonal terms involve sums of the form
\[
\sum_{m,n\asymp q^{\frac{3}{2}}} \frac{\tau_3(m)\tau_3(n)}{\sqrt{mn}} \frac{2}{\phi(q)} \sum_{\substack{\chi \mod{q}\\ \chi(-1)=(-1)^k}} \sum_{c} S_\chi(m,n;cq) J_{k-1}\fracp{4\pi \sqrt{mn}}{cq},
\]
where $S_\chi(m,n;cq)$ is the Kloosterman sum defined in (\ref{eq:KloostermanSum}) and $J_{k-1}$ is the usual Bessel function of order $k-1$.

The most important range for $c$ is in the transition region for the Bessel function, i.e. $c\asymp q^\frac{1}{2}$. To focus on this region, we truncate the sum in $c$ using the Weil estimate for Kloosterman sums. The details of this truncation are given in Section \ref{sec:Truncations}. The conductor of the Kloosterman sum is then essentially of size $cq\asymp q^\frac{3}{2}$. To understand the correlations between the Kloosterman sums and the Bessel functions, we apply harmonic analysis in the form of the Voronoi formula of Ivi\'{c} \citep{Ivic}.  Before doing so, we reduce the conductor in the Kloosterman sums by taking advantage of the average over $\chi$. The conductor lowering trick of \cite{CL} (see Lemma \ref{lem:ConductorLowering}) produces new Kloosterman sums of the form
\[
e\fracp{m+n}{cq} \sumstar_{x(c)} e\fracp{\bar{q}(x-1)m + \bar{q}(\bar{x}-1)n}{c},
\]
where the conductor is now reduced to $c\asymp q^\frac{1}{2}$ and the exponential in front may be treated as a smooth function with small derivatives. Applying the Voronoi formula then produces a single main term and eight error terms. The details of these transformations are given in Section \ref{sec:Voronoi}. The main term is estimated in Section \ref{sec:REstimate}, and it is here that we require a more delicate analysis of the Laurent series coefficients $D_{-i}$ of the third-order Estermann zeta function $E_3\pth{s,\frac{\lambda}{\eta}}$ (see (\ref{eq:EstermannDef}), (\ref{eq:E3Laurent}), and (\ref{eq:Ddef})). This is accomplished via Lemma \ref{lem:Dbound}, in which we improve the trivial estimate $D_{-i}(\eta) \ll \eta^{-1+\ep}$ to
\[
D_{-i}(\eta) \ll \frac{\tau_2(\eta)(\log \eta)^{3-i}}{\eta}.
\]

In order to apply this estimate without losing too much from the triangle inequality, we need to suitably transform the sum. This is accomplished by an application of Poisson summation (in the form of Lemma \ref{lem:BesselSplit}), along with some identities involving integrals of Bessel functions. The details of these transformations are given in Subsection \ref{subsec:Fourier}. 

After the above transformations, we would like to extract the contribution from the residue of $\zeta(s)$ at $s=1$. However, it turns out that the Laurent coefficients $D_{-i}$ are not multiplicative for $i=1,2$. Consequently, we require a somewhat delicate analysis of certain complex-valued arithmetic functions (see Proposition \ref{prop:FinalPiece}), in contrast to the simple contour shifting argument used in \cite{CL}. We then estimate the remaining sum by elementary means to obtain a final estimate of $(\log q)^9$ for the main term. These computations are given in Subsection \ref{subsec:RemainingSeries}.

Finally, the eight error terms (i.e. the dual sums arising from Voronoi summation) are estimated in Section \ref{sec:EEstimates}. The main aspect of the calculations in this section is that the dual sums are short. This is precisely the reason for reducing the conductor in the Kloosterman sums using Lemma \ref{lem:ConductorLowering}. The details of these calculations are standard but technical and follow closely the arguments of \citep{CL}.

\section{Approximate Functional Equation}\label{sec:AFE}

As is standard in such problems, we begin with an approximation functional equation for $L(f,1/2)^3$. The derivation of this is standard. For our purposes, it suffices to cite equation (2.5) of \citep{Djankovic}, which is
\[
\begin{aligned}
	L(f,1/2)^3 &= \sum_{a\geq 1} \sum_{b\geq 1} \sum_{n\geq 1} \frac{\mu(a)\chi(b)\tau_3(b)\lambda_f(an)\tau_3(n)}{(a^3b^2n)^{\frac{1}{2}}} U\fracp{a^3b^2n}{q^\frac{3}{2}} \\
	&+(i^k\bar{\eta_f})^3 \sum_{a\geq 1} \sum_{b\geq 1} \sum_{n\geq 1} \frac{\mu(a)\bar{\chi}(b)\tau_3(b)\bar{\lambda_f}(an)\tau_3(n)}{(a^3b^2n)^{\frac{1}{2}}}U\fracp{a^3b^2n}{q^\frac{3}{2}},
\end{aligned}
\]
where
\[
U(y) = \frac{1}{2\pi i} \int\limits_{(2)} y^{-s} \gamma^3(s) \pth{e^{s^2}}^3 \frac{ds}{s}, \qquad \gamma(s) = (2\pi)^{-s}\frac{\Gamma(\fract{k}{2}+s)}{\Gamma(\fract{k}{2})}.
\]
Here we have chosen the specific function $e^{s^2}$ to ensure that $U(y)$ is real when $y$ is real. This is mainly for notational simplicity. The function $U$ satisfies
\begin{equation}\label{eq:AFEUbounds}
	\begin{aligned}
		U(y) &\ll (1+y)^{-A},\\
		U(y) &= 1 + O(y^A) \quad \text{as}\ y\to 0
	\end{aligned}
\end{equation}
for any $A > 1$. Applying Cauchy's inequality, we have
\begin{equation}\label{eq:FirstCauchy}
\colm(q) \ll \frac{2}{\phi(q)} \sum_{\substack{\chi(q)\\ \chi(-1)=(-1)^k}} \sumh_{f\in\colh_\chi} \sumabs{\sum_{a,b,n\geq 1} \frac{\mu(a)\tau_3(b)\tau_3(n)\chi(ab)\lambda_f(an)}{(a^3b^2n)^{\frac{1}{2}}} U\fracp{a^3b^2n}{q^\frac{3}{2}}}^2.
\end{equation}
Expanding the square and rearranging, we obtain
\[
\begin{aligned}
	\colm(q) \ll &\summany_{\substack{a_1,b_1,n \geq 1\\a_2,b_2,m\geq 1}} \frac{\mu(a_1)\mu(a_2)\tau_3(b_1)\tau_3(n)\tau_3(b_2)\tau_3(m)}{(a_1^3b_1^2n)^{\frac{1}{2}}(a_2^3b_2^2m)^{\frac{1}{2}}} U\fracp{a_1^3b_1^2n}{q^\frac{3}{2}}U\fracp{a_2^3b_2^2m}{q^\frac{3}{2}} \\
	&\hspace{1cm} \times\sumpth{\frac{2}{\phi(q)} \sum_{\substack{\chi(q)\\ \chi(-1)=(-1)^k}} \chi(a_1b_1) \bar{\chi}(a_2b_2)}\sumpth{\sumh_{f\in\colh_\chi} \lambda_f(a_1n)\bar{\lambda_f}(a_2m)}.
\end{aligned}
\]
Note that by (\ref{eq:AFEUbounds}), the terms with $a_1^3b_1^2n, a_2^3b_2^2m  \gg q^{\frac{3}{2}+\ep}$ give a contribution of $q^{-2022}$. 

\section{Orthogonality and the Diagonal Contribution}\label{sec:Diagonal}
We now apply the orthogonality relations for $\chi$ and $\lambda_f$ given by the following lemma.
\begin{lem}\label{lem:Orthogonality}
	The orthogonality relation for Dirichlet characters is
	\[
	\frac{2}{\phi(q)} \sum_{\substack{\chi(q)\\ \chi(-1)=(-1)^k}} \chi(m)\bar{\chi}(n) = \begin{cases}
		1 & \text{if $m\equiv n(q)$, $(mn,q)=1$},\\
		(-1)^k & \text{if $m\equiv -n(q)$, $(mn,q)=1$},\\
		0 & \text{otherwise.}
	\end{cases}
	\]
	Petersson's formula gives
	\[
	\sumh_{f\in\colh_\chi} \lambda_f(n)\bar{\lambda}_f(m) = \delta_{m=n} + \sigma_\chi(m,n),
	\]
	where
	\[
	\sigma_\chi(m,n) = 2\pi i^{-k} \sum_{c=1}^\infty (cq)^{-1} S_\chi(m,n;cq) J_{k-1}\pth{\frac{4\pi}{cq}\sqrt{mn}}
	\]
	and $S_\chi$ is the Kloosterman sum defined by
	\begin{equation}\label{eq:KloostermanSum}
		S_\chi(m,n,cq) = \sumstar_{a(cq)} \chi(a) e\fracp{am+\bar{a}n}{cq}.
	\end{equation}
	Here $\sumstar$ denotes a sum over residues $a$ with $(a,cq)=1$ and $\bar{a}$ satisfies $a\bar{a}\equiv 1(cq)$.
\end{lem}

Lemma \ref{lem:Orthogonality} gives
\[
\colm(q) \ll \cold + \colo\cold,
\]
where $\colo\cold$ is given by (\ref{eq:OffDiagonal}) and $\cold=\cold_+ + \cold_-$ with
\[
\cold_+ = \summany_{\substack{a_1,b_1,n,a_2,b_2,m\geq 1\\ a_1b_1\equiv a_2b_2 (q)\\ a_1n=a_2m\\ (a_1b_1a_2b_2,q)=1}} \frac{\mu(a_1)\mu(a_2)\tau_3(b_1)\tau_3(n)\tau_3(b_2)\tau_3(m)}{(a_1^3b_1^2n)^{\frac{1}{2}}(a_2^3b_2^2m)^{\frac{1}{2}}} U\fracp{a_1^3b_1^2n}{q^\frac{3}{2}}U\fracp{a_2^3b_2^2m}{q^\frac{3}{2}}
\]
and $\cold_-$ is the same sum but multiplied by $(-1)^k$ with the condition $a_1b_1\equiv a_2b_2 \mod{q}$ replaced by $a_1b_1\equiv -a_2b_2 \mod{q}$. The only relevant case is when $a_1b_1=a_2b_2$ in $\cold_+$, since in the other cases we have $a_1b_1 \geq q/4$ or $a_2b_2 \geq q/4$, which means that $a_1^3b_1^2n \gg q^2$ or $a_2^3b_2^2m\gg q^2$. Thus
\begin{equation}\label{eq:diagonal}
	\cold = \summany_{\substack{a_1,b_1,n,a_2,b_2,m\geq 1\\ a_1b_1= a_2b_2\\ a_1n=a_2m\\ (a_1b_1a_2b_2,q)=1}} \frac{\mu(a_1)\mu(a_2)\tau_3(b_1)\tau_3(n)\tau_3(b_2)\tau_3(m)}{(a_1^3b_1^2n)^{\frac{1}{2}}(a_2^3b_2^2m)^{\frac{1}{2}}} U\fracp{a_1^3b_1^2n}{q^\frac{3}{2}}U\fracp{a_2^3b_2^2m}{q^\frac{3}{2}} +O(q^{-2022}).
\end{equation}
Neglecting the error term, we open the factors of $U$ and write
\begin{equation}\label{eq:DiagonalIntegral}
	\cold = \frac{1}{(2\pi i)^2} \int\limits_{(2)} \int\limits_{(2)} q^{\frac{3}{2}(s_1+s_2)} \mathscr{D}(1+s_1+s_2) \gamma^3(s_1)\gamma^3(s_2) e^{3(s_1^2+s_2^2)} \frac{ds_1}{s_1}\ \frac{ds_2}{s_2}.
\end{equation}
where
\[
\mathscr{D}(s) = \sum_{\substack{a_1,b_1,n\geq 1\\(a_1b_1,q)=1}} \frac{\mu(a_1)\tau_3(b_1)\tau_3(n)}{(a_1^3b_1^2n)^s} \sum_{\substack{a_2,b_2,m\\ a_2b_2=a_1b_1\\ a_2m=a_1n\\ (a_2b_2,q)=1}} \mu(a_2)\tau_3(b_2)\tau_3(m).
\]
We write $\mathscr{D}(s)$ as the Euler product
\[
\mathscr{D}(s) = \prod_{p} \mathscr{D}_p(s),
\]
where 
\[
\mathscr{D}_p(s) = \summany_{\substack{a_1,b_1,n,a_2,b_2,m\geq 0\\ a_1+b_1=a_2+b_2\\ a_1+n=a_2+m}} \frac{\mu(p^{a_1})\mu(p^{a_2})\tau_3(p^{b_1})\tau_3(p^{n})\tau_3(p^{b_2})\tau_3(p^{m})}{p^{s(3a_1+2b_1+n)}} = 1+\frac{9}{p^s} + \cdots
\]
for $p\neq q$ and
\[
\mathscr{D}_q(s) = \sum_{n\geq 0} \frac{\tau_3(q^n)^2}{q^{sn}} = 1 + \frac{9}{q^s} + \cdots.
\]
Thus
\[
\mathscr{D}(s) = \zeta^9(s) H(s)
\]
for some $H(s)$ that is analytic for $\Re s > 1/2$. After the change of variables $u=s_1+s_2$, $s=s_2$, we have
\[
\cold = \frac{1}{(2\pi i)^2} \int\limits_{(2)} \int\limits_{(4)} q^{\frac{3}{2}u} \zeta^9(1+u)H(1+u) \gamma^3(u-s)\gamma^3(s) e^{3(u^2-2us)} \frac{du}{u-s}\ \frac{ds}{s}.
\]
The rapid decay of $\gamma(s)$ and $e^{3u^2}$ on vertical lines allows us to move the line of integration in $s$ to $\Re s = -1$ and the integration in $u$ to $\Re u = -\frac{1}{2}+\ep$. In doing so, we pass a simple pole at $s=0$ and poles of orders 9 and 10 at $u=0$. Thus
\[
\cold = R_1 + R_2 + E_1 + E_2,
\]
where
\[
\begin{aligned}
	R_1 &= \res{u=0} \left[q^{\frac{3}{2}u} \zeta^9(1+u)H(1+u) \gamma^3(u) e^{3u^2} u^{-1}\right], \\
	R_2 &= \frac{1}{2\pi i}\int\limits_{(-1)} \gamma^3(s) \res{u=0} \left[q^{\frac{3}{2}u} \zeta^9(1+u)H(1+u) \gamma^3(u-s) e^{3u^2-2us} \frac{1}{u-s}\right] \frac{ds}{s},\\
	E_1 &= \frac{1}{2\pi i} \int\limits_{(-\frac{1}{2}+\ep)} q^{\frac{3}{2}u} \zeta^9(1+u)H(1+u) \gamma^3(u)  e^{3u^2} \frac{du}{u},\\
	E_2 &= \frac{1}{(2\pi i)^2} \int\limits_{(-1)} \int\limits_{(-\frac{1}{2}+\ep)} q^{\frac{3}{2}u} \zeta^9(1+u)H(1+u) \gamma^3(u-s)\gamma^3(s) e^{3(u^2-2us)} \frac{du}{u-s}\ \frac{ds}{s}.
\end{aligned}
\]
Using Stirling's formula and the rapid decay of $e^{3u^2}$, we see that
\[
E_1,E_2 \ll q^{-3/4+\ep}.
\]
A straightforward calculation shows that
\[
R_1 \asymp (\log q)^9.
\]
In $R_2$, the leading order term of the residue (in terms of $q$) is of the form
\[
\frac{C(\log q)^8}{2\pi i}\int\limits_{(-1)} \gamma^3(s)\gamma^3(-s)  \frac{ds}{s^2} \asymp (\log q)^8.
\]
We deduce that $R_2 \asymp (\log q)^8$, from which it follows that
\[
\cold \asymp (\log q)^9.
\]

We proceed now to our treatment of the off-diagonal terms, which constitutes the remainder of the proof

\section{Truncation of the The Off-Diagonal Terms}\label{sec:Truncations}

The off-diagonal contribution is
\[
\begin{aligned}
	\colo\cold&=\summany_{\substack{a_1,b_1,n\geq 1\\ a_2,b_2,m\geq 1}} \frac{\mu(a_1)\mu(a_2)\tau_3(b_1)\tau_3(n)\tau_3(b_2)\tau_3(m)}{(a_1^3b_1^2n)^{\frac{1}{2}}(a_2^3b_2^2m)^{\frac{1}{2}}} U\fracp{a_1^3b_1^2n}{q^\frac{3}{2}}U\fracp{a_2^3b_2^2m}{q^\frac{3}{2}} \\	&\hspace{1cm} \times\sumpth{\frac{2}{\phi(q)} \sum_{\substack{\chi(q)\\ \chi(-1)=(-1)^k}} \chi(a_1b_1) \bar{\chi}(a_2b_2)} \sumpth{2\pi i^{-k} \sum_{c\geq1} (cq)^{-1}\sum_{a\bar{a}\equiv 1 (cq)} \chi(a)e\fracp{aa_2m+\bar{a}a_1n}{cq}} \\
	&=\frac{1}{q}\summany_{\substack{a_1,b_1,n\geq 1\\ a_2,b_2,m\geq 1}} \frac{\mu(a_1)\mu(a_2)\tau_3(b_1)\tau_3(n)\tau_3(b_2)\tau_3(m)}{(a_1^3b_1^2n)^{\frac{1}{2}}(a_2^3b_2^2m)^{\frac{1}{2}}} U\fracp{a_1^3b_1^2n}{q^\frac{3}{2}}U\fracp{a_2^3b_2^2m}{q^\frac{3}{2}} \\	&\hspace{1cm} \times 2\pi i^{-k} \sum_{c\geq1} \frac{1}{c}\sum_{a\bar{a}\equiv 1 (cq)} e\fracp{aa_2m+\bar{a}a_1n}{cq} \sumpth{\frac{2}{\phi(q)} \sum_{\substack{\chi(q)\\ \chi(-1)=(-1)^k}} \chi(aa_1b_1) \bar{\chi}(a_2b_2)}.
\end{aligned}
\]
As in \citep{CL}, we introduce the operator $\colk g = i^{-k} g + i^k \bar{g}$ for notational convenience. Let $f$ be a smooth function supported on $[\half,3]$ such that
\[
\sum_{j\in\Z} f\fracp{t}{2^j} = 1
\]
for all $ t\geq 0$. Inserting two of these dyadic partitions of unity into the sums over $n,m$ and using the orthogonality relations for $\chi$, we find that the off-diagonal contribution is
\begin{equation}\label{eq:OffDiagonal}
	\begin{aligned}
		&\frac{2\pi}{q} \sumfour_{\substack{a_1,b_1,a_2,b_2\geq 1\\ (a_1a_2b_1b_2,q)=1}} \frac{\mu(a_1)\mu(a_2)\tau_3(b_1)\tau_3(b_2)}{(a_1^3b_1^2a_2^3b_2^2)^{\frac{1}{2}}} \sumd_{N} \sumd_{M} \sumtwo_{n,m\geq 1} \frac{\tau_3(n)\tau_3(m)}{(nm)^{\frac{1}{2}}} \\
		&\hspace{2cm}\times \sum_{c\geq 1} \frac{1}{c} \colg(\bolda,\boldb,\boldn,\boldN,c) \colk \sumstar_{\substack{a(cq)\\ a\equiv \bar{a_1b_1}a_2b_2(q)}} e\fracp{aa_2m+\bar{a}a_1n}{c},
	\end{aligned}
\end{equation}
where
\[
\mathcal{G}(\bolda,\boldb,\boldn,\boldN,c) = U\fracp{a_1^3b_1^2n}{q^\frac{3}{2}}U\fracp{a_2^3b_2^2m}{q^\frac{3}{2}} f\fracp{n}{N}f\fracp{m}{M} J_{k-1}\pth{\frac{4\pi}{cq}\sqrt{a_1na_2m}}.
\]
Here $\sumd_N$ denotes a dyadic sum over $N = 2^j$.

We now truncate the sum in $c$. Letting $C=q^{-\frac{2}{3}}\sqrt{a_1a_2NM}$, we write
\[
\colo\cold = \mathscr{M} + \mathscr{K}_1 + \mathscr{K}_2,
\]
where
\begin{equation}\label{eq:scrKdef}
	\mathscr{K}_i = \frac{2\pi }{q} \sumfour_{\substack{a_1,b_1,a_2,b_2\geq 1\\ (a_1a_2b_1b_2,q)=1}} \frac{\mu(a_1)\mu(a_2)\tau_3(b_1)\tau_3(b_2)}{(a_1^3b_1^2a_2^3b_2^2)^{\frac{1}{2}}} \sumd_{N,M} \mathscr{S}_i(\boldsymbol{a}, \boldsymbol{b}, \boldsymbol{N})
\end{equation}
with
\[
\begin{aligned}
	\mathscr{S}_1(\boldsymbol{a}, \boldsymbol{b}, \boldsymbol{N}) &= \sum_{\substack{c\geq 1\\ q\mid c}} \frac{1}{c} \sum_{n,m\geq 1} \frac{\tau_3(n)\tau_3(m)}{(nm)^{\frac{1}{2}}} \mathcal{F}(\bolda,\boldb,\boldn,\boldN,c), \\
	\mathscr{S}_2(\boldsymbol{a}, \boldsymbol{b}, \boldsymbol{N}) &= \sum_{\substack{c > C\\ (c,q) = 1}} \frac{1}{c} \sum_{n,m\geq 1} \frac{\tau_3(n)\tau_3(m)}{(nm)^{\frac{1}{2}}} \mathcal{F}(\bolda,\boldb,\boldn,\boldN,c),
\end{aligned}
\]
and
\[
\mathcal{F}(\bolda,\boldb,\boldn,\boldN,c) = \mathcal{G}(\bolda,\boldb,\boldn,\boldN,c) \colk \sumstar_{\substack{a(cq)\\ a\equiv \bar{a_1b_1}a_2b_2(q)}} e\fracp{aa_2m+\bar{a}a_1n}{cq}.
\]
The quantity $\mathscr{M}$ is defined like $\mathscr{K}_2$, except with the condition $c \leq C$ replaced by $c > C$. We now prove the following proposition.

\begin{prop}\label{prop:Truncation}
	For $C = q^{-\frac{2}{3}} \sqrt{a_1a_2NM}$, we have
	\[
	\mathscr{K}_1 + \mathscr{K}_2 \ll q^{-\frac{5}{12}+\ep}.
	\]
\end{prop}

For the proof of this proposition and for our arguments in Section \ref{sec:EEstimates}, we will need several properties of the $J$-Bessel functions. These are summarized in the following lemma. These results are standard and can all be found in \cite{Watson}.
\begin{lem}\label{lem:Bessel}
	We have
	\begin{equation}\label{eq:BesselW}
		J_{k-1}(2\pi x) = \frac{1}{\pi \sqrt{x}} \pth{W(2\pi x) e\pth{x-\frac{k}{4}+\frac{1}{8}} + \bar{W}(2\pi x) e\pth{-x+\frac{k}{4}-\frac{1}{8}}},
	\end{equation}
	where $W^{(j)}(x) \ll_{j,k} x^{-j}$. Moreover,
	\begin{equation}\label{eq:BesselSeries}
		J_{k-1}(2\pi x) = \sum_{\ell=0}^\infty (-1)^\ell \frac{x^{2\ell+k+1}}{\ell! (\ell+k-1)!},
	\end{equation}
	and
	\begin{equation}\label{eq:BesselMin}
		J_{k-1}(x) \ll \min(x^{-\frac{1}{2}}, x^{k-1}).
	\end{equation}
\end{lem}

\begin{proof}[Proof of Proposition \ref{prop:Truncation}]
	To treat $\mathscr{K}_1$, we begin by writing
	\[
	\mathscr{S}_1(\boldsymbol{a}, \boldsymbol{b}, \boldsymbol{N}) = \sum_{r=1}^\infty \frac{1}{q^r} \sum_{\substack{c\geq 1\\ (c,q)=1}} \frac{1}{c} \sum_{n,m\geq 1} \frac{\tau_3(n)\tau_3(m)}{(nm)^{\frac{1}{2}}} \mathcal{F}(\bolda,\boldb,\boldn,\boldN,cq^r).
	\]
	For a fixed $r$, we use the Chinese Remainder Theorem and the Weil bound to see that the modulus of the Kloosterman sum in $\colf$ is
	\[
	\begin{aligned}
		\sumabs{\sumstar_{\substack{a(cq^{r+1})\\ a\equiv \bar{a_1b_1}a_2b_2(q)}} e\fracp{aa_2m+\bar{a}a_1n}{cq^{r+1}}} &= \sumabs{\sumstar_{x\mod{q^r}} e\fracp{xa_2m+\bar{x}a_1n}{q^r}} \sumabs{\sumstar_{y\mod{c}} e\fracp{ya_2m+\bar{y}a_1n}{c}} \\
		&\ll (cq^r)^{\frac{1}{2}+\ep} \sqrt{(a_1n,a_2m,c)(n,m,q^r)}.
	\end{aligned}
	\]
	From (\ref{eq:BesselMin}), we have
	\[
	J_{k-1} \pth{\frac{4\pi}{cq^{r+1}} \sqrt{a_1na_2m}} \ll \fracp{\sqrt{a_1a_2NM}}{cq^{r+1}}^2,
	\]
	and so
	\[
	\begin{aligned}
		\mathscr{S}_1(\boldsymbol{a}, \boldsymbol{b}, \boldsymbol{N}) &\ll a_1a_2(NM)^\frac{1}{2}\sum_{r=1}^\infty \frac{1}{q^{\frac{5r}{2}+2-\ep}} \sum_{\substack{c\geq 1\\ (c,q)=1}} \frac{1}{c^{\frac{5}{2}-\ep}} \sum_{\substack{n\asymp N\\ m\asymp M}} \tau_3(n)\tau_3(m) \sqrt{(a_1n,a_2m,c)(n,m,q^r)}\\
		&\ll a_1a_2(NM)^\frac{1}{2}\sum_{r=2}^\infty \frac{1}{q^{2r-\ep}} \sum_{c\geq 1} \frac{1}{c^{2-\ep}} \sum_{\substack{n\asymp N\\ m\asymp M}} \tau_3(n)\tau_3(m)  \\
		&\ll \frac{a_1a_2(NM)^\frac{3}{2}}{q^{4-\ep}}.
	\end{aligned}
	\]
	Here we have bounded the gcds by $c$ and $q^r$, respectively. Returning to (\ref{eq:scrKdef}), we conclude that
	\[
		\mathscr{K}_1 \ll \frac{1}{q^{5-\ep}} \underset{a_1^3b_1^2N,a_2^3b_2^2M\ll q^{\frac{3}{2}+\epsilon}}{\sumfour_{a_1,b_1,a_2,b_2\geq 1} \sumd_{N,M}} \frac{\tau_3(b_1)\tau_3(b_2)}{(a_1a_2)^{\frac{1}{2}}b_1b_2}  (NM)^\frac{3}{2} \ll q^{-\frac{1}{2}+\ep}.
	\]
	We now turn to $\mathscr{K}_2$. Again by the Chinese Remainder Theorem and the Weil bound, the modulus of the Kloosterman sum in $\colf$ is
	\[
	\sumabs{\sumstar_{\substack{a(cq)\\ a\equiv \bar{a_1b_1}a_2b_2(q)}} e\fracp{aa_2m+\bar{a}a_1n}{cq}} = \sumabs{\sumstar_{y\mod{c}} e\fracp{ya_2m+\bar{y}a_1n}{c}} \ll c^{\frac{1}{2}+\ep} \sqrt{(a_1n,a_2m,c)}.
	\]
	Using (\ref{eq:BesselMin}) once more, we see that
	\[
	\begin{aligned}
		\mathscr{S}_2(\boldsymbol{a}, \boldsymbol{b}, \boldsymbol{N}) &\ll q^\ep(NM)^{\frac{k}{2}-1}  \sum_{c > C} c^{\ep-\frac{1}{2}} \fracp{\sqrt{a_1a_2}}{cq}^{k-1} \sum_{\substack{n\asymp N\\ m\asymp M}}  \sqrt{(a_1n,a_2m,c)} \\
		&\ll q^\ep(NM)^{\frac{k}{2}-1}  \sum_{c > C} c^{\ep-\frac{1}{2}} \fracp{\sqrt{a_1a_2}}{cq}^{k-1} \sum_{d\mid c} \sqrt{d} \sum_{\substack{n\asymp N\\ m\asymp M \\ (a_1n,a_2m,c)=d}}  1 \\
		&\ll q^\ep(NM)^{\frac{k}{2}-1}  \sum_{c > C} c^{\ep-\frac{1}{2}} \fracp{\sqrt{a_1a_2}}{cq}^{k-1} \sum_{d\mid c} \sqrt{d} \sum_{\substack{n\asymp N \\d\mid a_1n} } \sum_{\substack{m\asymp M \\d\mid a_2m}} 1 \\
		&\ll q^\ep(NM)^{\frac{k}{2}} \sum_{c > C} c^{\ep-\frac{1}{2}} \fracp{\sqrt{a_1a_2}}{cq}^{k-1} \sum_{d\mid c} \frac{(d,a_1)(d,a_2)}{d^\frac{3}{2}} \\
		&\ll q^\ep\frac{(a_1a_2NM)^{\frac{k}{2}}}{q^{k-1}} C^{\frac{3}{2}-k+\ep}  \ll q^\ep\frac{(a_1a_2NM)^{\frac{3}{4}}}{q^{\frac{5}{3}} },
	\end{aligned}
	\]
	so long as $k \geq 5$. On the fifth line above we have used the estimate $(d,a_i) \leq \sqrt{da_i}$. Once again returning to (\ref{eq:scrKdef}), we conclude that
	\[
	\mathscr{K}_2 \ll \frac{1}{q^{\frac{8}{3}-\ep}} \underset{a_1^3b_1^2N,a_2^3b_2^2M\ll q^{\frac{3}{2}+\epsilon}}{\sumfour_{a_1,b_1,a_2,b_2\geq 1} \sumd_{N,M}} \frac{\tau_3(b_1)\tau_3(b_2)}{(a_1a_2)^{\frac{3}{4}}b_1b_2}  (NM)^\frac{3}{4} \ll q^{-\frac{5}{12}+\ep}.
	\]
\end{proof}

\section{Voronoi Summation}\label{sec:Voronoi}
It remains to estimate
\begin{equation}\label{eq:MainOffDiagonal}
	\begin{aligned}
		&\frac{2\pi }{q} \sumfour_{\substack{a_1,b_1,a_2,b_2\geq 1\\ (a_1a_2b_1b_2,q)=1}} \frac{\mu(a_1)\mu(a_2)\tau_3(b_1)\tau_3(b_2)}{(a_1^3b_1^2a_2^3b_2^2)^{\frac{1}{2}}} \sumd_{N,M} \sum_{n,m\geq 1} \frac{\tau_3(n)\tau_3(m)}{(nm)^{\frac{1}{2}}} U\fracp{a_1^3b_1^2n}{q^\frac{3}{2}}U\fracp{a_2^3b_2^2m}{q^\frac{3}{2}}\\
		&\hspace{1cm}\times f\fracp{n}{N}f\fracp{m}{M}\sum_{c\leq C} \frac{1}{c} J_{k-1}\pth{\frac{4\pi}{cq}\sqrt{a_1na_2m}}\colk \sumstar_{\substack{a(cq)\\ a\equiv \bar{a_1b_1}a_2b_2(q)}} e\fracp{aa_2m+\bar{a}a_1n}{cq}.
	\end{aligned}
\end{equation}
Before applying Voronoi summation, we reduce the conductor in the Kloosterman sum using Lemma 5.3 of \citep{CL}, which we cite in the following form.

\begin{lem}\label{lem:ConductorLowering}
Let $c,q,m,n$ be positive integers with $(cmn,q)=1$, and let
\[
Y(u,v) = \sumstar_{\substack{a(cq)\\ a\equiv \bar{m}n(q)}} e\fracp{au+\bar{a}v}{cq}.
\]
Then
\[
Y(u,v) = e\fracp{n^2u + m^2v}{cqmn} \sumstar_{x(c)} e\fracp{\bar{q}(mx-n)u}{mc}e\fracp{\bar{q}(n\bar{x}-m)v}{nc},
\]
where $\bar{q}$ denotes the inverse of $q \mod{c}$.
\end{lem}

\noindent Applying Lemma \ref{lem:ConductorLowering} with $m=a_1b_1$, $n=a_2b_2$, the expression in (\ref{eq:MainOffDiagonal}) is
\[
\begin{aligned}
	&\frac{2\pi }{q} \sumfour_{\substack{a_1,b_1,a_2,b_2\geq 1\\ (a_1a_2b_1b_2,q)=1}} \frac{\mu(a_1)\mu(a_2)\tau_3(b_1)\tau_3(b_2)}{(a_1^3b_1^2a_2^3b_2^2)^{\frac{1}{2}}} \sumd_{N,M} \sum_{n,m\geq 1} \frac{\tau_3(n)\tau_3(m)}{(nm)^{\frac{1}{2}}} U\fracp{a_1^3b_1^2n}{q^\frac{3}{2}}U\fracp{a_2^3b_2^2m}{q^\frac{3}{2}}\\
	&\hspace{1cm}\times f\fracp{n}{N}f\fracp{m}{M}\sum_{\substack{c\leq C\\ (c,q)=1}} \frac{1}{c} J_{k-1}\pth{\frac{4\pi}{cq}\sqrt{a_1na_2m}}\colk e\fracp{(a_1b_1)^2a_1n + (a_2b_2)^2a_2m}{cqa_1b_1a_2b_2} \\
	&\hspace{1cm}\times \sumstar_{x(c)} e\fracp{\bar{q}(a_2b_2\bar{x} - a_1b_1)a_1n}{a_2b_2c} e\fracp{\bar{q}(a_1b_1x - a_2b_2)a_2m}{a_1b_1c}  \\
	&=\frac{2\pi }{q} \sumfour_{\substack{a_1,b_1,a_2,b_2\geq 1\\ (a_1a_2b_1b_2,q)=1}} \frac{\mu(a_1)\mu(a_2)\tau_3(b_1)\tau_3(b_2)}{(a_1^3b_1^2a_2^3b_2^2)^{\frac{1}{2}}} \sumd_{N,M} \sum_{c \leq C} \frac{1}{c}\sumstar_{x(c)} \colk\ \cols(c,x), \\
\end{aligned}
\]
where
\[
\begin{aligned}
	\cols(c,x) &= \cols(c,x;\bolda,\boldb,\boldN) \\
	&= \sum_{n\geq 1} \tau_3(n) e\pth{\frac{\bar{q}(a_2b_2\bar{x} - a_1b_1)a_1n}{a_2b_2c}} \sum_{m\geq 1} \tau_3(m) e\pth{\frac{\bar{q}(a_1b_1x - a_2b_2)a_2m}{a_1b_1c}} \\
	&\hspace{.25in}\times F_1(n) F_2(m) J_{k-1}\pth{\frac{4\pi}{cq}\sqrt{a_1na_2m}},
\end{aligned}
\]
and
\[
\begin{aligned}
	F_1(y) &= y^{-\frac{1}{2}} f\fracp{y}{N} e\pth{\frac{a_1^2b_1y}{cqa_2b_2}} U\fracp{a_1^3b_1^2y}{q^\frac{3}{2}}, \\
	F_2(y) &= y^{-\frac{1}{2}} f\fracp{y}{M} e\pth{\frac{a_2^2b_2y}{cqa_1b_1}} U\fracp{a_2^3b_2^2y}{q^\frac{3}{2}}.
\end{aligned}
\]
We write
\begin{equation}\label{eq:EtaLambdaDef}
	\begin{aligned}
		\frac{\lambda_1}{\eta_1} &= \frac{\bar{q}(a_2b_2\bar{x} - a_1b_1)a_1}{a_2b_2c}, \\
		\frac{\lambda_2}{\eta_2} &= \frac{\bar{q}(a_1b_1x - a_2b_2)a_2}{a_1b_1c},
	\end{aligned}
\end{equation}
where $(\lambda_1,\eta_1) = (\lambda_2,\eta_2) = 1$. Thus
\[
\cols(c,x) = \sum_{n\geq 1} F_1(n)\tau_3(n) e\fracp{\lambda_1n}{\eta_1} \sum_{m\geq 1} F_2(m)\tau_3(m) e\fracp{\lambda_2m}{\eta_2} J_{k-1}\pth{\frac{4\pi}{cq}\sqrt{a_1na_2m}}.
\]
We also define
\[
\begin{aligned}
	\colu(c;\boldy) &= F_1(y_1)F_2(y_2) \\
	&=  \frac{1}{(y_1y_2)^\frac{1}{2}} f\fracp{y_1}{N} f\fracp{y_2}{M} e\pth{\frac{a_1^2b_1y_1}{cqa_2b_2}+\frac{a_2^2b_2y_2}{cqa_1b_1}} U\fracp{a_1^3b_1^2y_1}{q^\frac{3}{2}}U\fracp{a_2^3b_2^2y_2}{q^\frac{3}{2}}.
\end{aligned}
\]
We now apply Theorem 2 of \citep{Ivic} with the same notations used there, except with $A_3^-$ in place of Ivic's $B_3$, first to the sum over $m$, and then to the sum over $n$, to obtain
\[
\cols(c,x) = \sum_{j=1}^9 \colt_j(c,x),
\]
where
\begin{equation}\label{eq:T1def}
	\begin{aligned}
		\colt_1(c,x) &= \res{s_1=1} \res{s_2=1} \bigg[ E_3\pth{s_1,\frac{\lambda_1}{\eta_1}} E_3\pth{s_2,\frac{\lambda_2}{\eta_2}} \\
		&\hspace{.5in}\times\int_0^\infty \int_0^\infty \colu(c;y_1,y_2) J_{k-1}\pth{\frac{4\pi}{cq}\sqrt{a_1a_2y_1y_2}} y_1^{s_1-1} y_2^{s_2-1} dy_1\ dy_2 \bigg],
	\end{aligned}
\end{equation}

\begin{equation}\label{eq:T2def}
	\begin{aligned}
		\colt_2(c,x) &= \frac{\pi^\frac{3}{2}}{\eta_2^3} \res{s_1=1} E_3\pth{s_1,\frac{\lambda_1}{\eta_1}} \int_0^\infty F_1(y_1) y_1^{s_1-1} \sum_{m\geq 1} A_3^+\pth{m,\frac{\lambda_2}{\eta_2}} \int_0^\infty F_2(y_2) U_3\fracp{\pi^3my_2}{\eta_2^3} \\
		&\hspace{1in} \times J_{k-1}\pth{\frac{4\pi}{cq}\sqrt{a_1a_2y_1y_2}} dy_1\ dy_2,
	\end{aligned}
\end{equation}
\begin{equation}\label{eq:T3def}
	\begin{aligned}
		\colt_3(c,x) &= \frac{i\pi^\frac{3}{2}}{\eta_2^3} \res{s_1=1} E_3\pth{s_1,\frac{\lambda_1}{\eta_1}} \int_0^\infty F_1(y_1) y^{s_1-1} \sum_{m\geq 1} A_3^-\pth{m,\frac{\lambda_2}{\eta_2}} \int_0^\infty F_2(y_2) V_3\fracp{\pi^3my_2}{\eta_2^3} \\
		&\hspace{1in} \times J_{k-1}\pth{\frac{4\pi}{cq}\sqrt{a_1a_2y_1y_2}} dy_1\ dy_2,
	\end{aligned}
\end{equation}
$\colt_4$ and $\colt_5$ are defined as $\colt_2$ and $\colt_3$, but one swaps all subscripts of 1 and 2, 
\begin{equation}\label{eq:T6def}
	\begin{aligned}
		\colt_6(c,x) &= \frac{\pi^3}{\eta_1^3\eta_2^3} \sum_{n\geq 1} \sum_{m\geq 1} A_3^+\pth{n,\frac{\lambda_1}{\eta_1}} A_3^+\pth{m,\frac{\lambda_2}{\eta_2}} \int_{0}^{\infty} \int_{0}^{\infty} F_1(y_1) F_2(y_2) \\
		&\hspace{1in}\times  U_3\fracp{\pi^3ny_1}{\eta_1^3} U_3\fracp{\pi^3my_2}{\eta_2^3} J_{k-1}\pth{\frac{4\pi}{cq} \sqrt{a_1y_1a_2y_2}} dy_1\ dy_2,
	\end{aligned}
\end{equation}
$\colt_7$ is defined similar to $\colt_6$ except we replace the leading coefficient by its negative, $A^+$ with $A^-$, and $U_3$ with $V_3$,
\begin{equation}\label{eq:T8def}
	\begin{aligned}
		\colt_8(c,x) &= \frac{i\pi^3}{\eta_1^3\eta_2^3} \sum_{n\geq 1} \sum_{m\geq 1} A_3^+\pth{n,\frac{\lambda_1}{\eta_1}} A_3^-\pth{m,\frac{\lambda_2}{\eta_2}} \int_{0}^{\infty} \int_{0}^{\infty} F_1(y_1) F_2(y_2) \\
		&\hspace{1in}\times  U_3\fracp{\pi^3ny_1}{\eta_1^3} V_3\fracp{\pi^3my_2}{\eta_2^3} J_{k-1}\pth{\frac{4\pi}{cq} \sqrt{a_1y_1a_2y_2}} dy_1\ dy_2,
	\end{aligned}
\end{equation}
and $\colt_9$ is defined as $\colt_8$, but one swaps all the subscripts of 1 and 2. Here and throughout, $E_3$ denotes the third-order Estermann zeta function:
\begin{equation}\label{eq:EstermannDef}
E_3\pth{s,\frac{\lambda}{\eta}} = \sum_{n=1}^\infty \frac{\tau_3(n)e\fracp{n\lambda}{\eta}}{n^s}.
\end{equation}
The only information we require about $E_3$ is its Laurent series expansion at $s=1$. For $(\lambda,\eta)=1$, this is
\begin{equation}\label{eq:E3Laurent}
	E_3\pth{s,\frac{\lambda}{\eta}} = \frac{D_{-3}(\eta)}{(s-1)^3} + \frac{D_{-2}(\eta)}{(s-1)^2} + \frac{D_{-1}(\eta)}{s-1} + D_0 + \cdots,
\end{equation}
where the coefficients $D_{-i}$ are given by 
\begin{equation}\label{eq:Ddef}
\begin{aligned}
D_{-3}(\eta) &= \frac{1}{\eta^2} \sum_{\alpha_1=1}^\eta \sum_{\alpha_2=1}^\eta \boldsymbol{1}(\eta | \alpha_1\alpha_2), \\
D_{-2}(\eta) &= \frac{1}{\eta^2} \sum_{\alpha_1=1}^\eta \sum_{\alpha_2=1}^\eta \boldsymbol{1}(\eta | \alpha_1\alpha_2)\pth{3\gamma_0\fracp{\alpha_1}{\eta} - 3\log \eta}, \\
D_{-1}(\eta) &= \frac{1}{\eta^2} \sum_{\alpha_1=1}^\eta \sum_{\alpha_2=1}^\eta \boldsymbol{1}(\eta | \alpha_1\alpha_2)\bigg(\frac{9}{2}(\log \eta)^2 - 9\gamma_0 \fracp{\alpha_1}{\eta}\log \eta  \\
&\hspace{2in} + 3\gamma_0\fracp{\alpha_1}{\eta} \gamma_0\fracp{\alpha_2}{\eta}+ 3\gamma_1\fracp{\alpha_1}{\eta}\bigg). \\
\end{aligned}
\end{equation}
Here $\boldsymbol{1}(\eta | \alpha_1\alpha_2)$ is 1 if $\eta$ divides $\alpha_1\alpha_2$ and 0 otherwise, and $\gamma_0,\gamma_1$ are generalized Stieltjes constants defined by
\begin{equation}\label{eq:StieltjesConstantsDef}
	\zeta(s,r) = \frac{1}{s-1} + \sum_{n=0}^\infty \gamma_n(r)(s-1)^n,
\end{equation}
where $\zeta(s,r)$ is the Hurwitz zeta function. For a proof of \eqref{eq:E3Laurent} and \eqref{eq:Ddef}, see Equation (2.13) of \cite{Djankovic}.

Theorem \ref{thm:Main} now follows from the following two propositions, which we prove in Sections \ref{sec:REstimate} and \ref{sec:EEstimates}, respectively.

\begin{prop}\label{prop:REstimate}
	Let
	\[
	\colr(q) = \frac{2\pi }{q} \sumfour_{\substack{a_1,b_1,a_2,b_2\geq 1\\ (a_1a_2b_1b_2,q)=1}} \frac{\mu(a_1)\mu(a_2)\tau_3(b_1)\tau_3(b_2)}{(a_1^3b_1^2a_2^3b_2^2)^{\frac{1}{2}}} \sumd_{N,M} \sum_{c \leq C} \frac{1}{c}\sumstar_{x(c)} \colk \colt_1(c,x).
	\]
	Then
	\[
	\colr(q) \ll (\log q)^{9}.
	\]
\end{prop}

\begin{prop}\label{prop:EEstimates}
	For $j=2,\ldots,9$, let
	\[
	\cole_j(q) = \frac{2\pi }{q} \sumfour_{\substack{a_1,b_1,a_2,b_2\geq 1\\ (a_1a_2b_1b_2,q)=1}} \frac{\mu(a_1)\mu(a_2)\tau_3(b_1)\tau_3(b_2)}{(a_1^3b_1^2a_2^3b_2^2)^{\frac{1}{2}}} \sumd_{N,M} \sum_{c \leq C} \frac{1}{c}\sumstar_{x(c)} \colk \colt_j(c,x).
	\]
	Then
	\[
	\cole_j(q) \ll q^{-\frac{1}{8}+\ep}
	\]
\end{prop}

\section{Proof of Proposition \ref{prop:REstimate}}\label{sec:REstimate}
Recall that
\[
\begin{aligned}
	\colt_1(c,x) &= \res{s_1=1} \res{s_2=1} \bigg[ E_3\pth{s_1,\frac{\lambda_1}{\eta_1}} E_3\pth{s_2,\frac{\lambda_2}{\eta_2}} \\
	&\hspace{.5in}\times\int_0^\infty \int_0^\infty \colu(c;y_1,y_2) J_{k-1}\pth{\frac{4\pi}{cq}\sqrt{a_1a_2y_1y_2}} y_1^{s_1-1} y_2^{s_2-1} dy_1\ dy_2 \bigg].
\end{aligned}
\]
To estimate the contribution from $\colt_1$, we first compute the residues via the Laurent series for $E_3$. For $j_1,j_2\geq 0$, let
\[
\coli_1(c;j_1,j_2) = \frac{1}{j_1!j_2!}\int_0^\infty \int_0^\infty \colu(c;y_1,y_2) (\log y_1)^{j_1}(\log y_2)^{j_2} J_{k-1}\pth{\frac{4\pi}{cq}\sqrt{a_1a_2y_1y_2}}  dy_1\ dy_2
\]
so that
\[
\begin{aligned}
	&\int_0^\infty \int_0^\infty \colu(c;y_1,y_2) J_{k-1}\pth{\frac{4\pi}{cq}\sqrt{a_1a_2y_1y_2}} y_1^{s_1-1} y_2^{s_2-1} dy_1\ dy_2\\
	&\hspace{2in}= \sum_{j_1=0}^\infty \sum_{j_2=0}^\infty \coli_1(c;j_1,j_2) (s_1-1)^{j_1}(s_2-1)^{j_2}.
\end{aligned}
\]
Using \eqref{eq:E3Laurent} and \eqref{eq:Ddef}, we have
\[
\colt_1(x,c) = \sum_{\substack{1\leq l_1,l_2\leq 3\\0\leq j_1,j_2\leq 2\\ j_i-l_i=-1}} D_{-l_1}(\eta_1)D_{-l_2}(\eta_2) \coli_1(c;j_1,j_2),
\]
and so
\begin{equation}\label{eq:ResidueContribution}
	\begin{aligned}
		\colr(q) &= \frac{2\pi }{q} \sumfour_{\substack{a_1,b_1,a_2,b_2\geq 1\\ (a_1a_2b_1b_2,q)=1}} \frac{\mu(a_1)\mu(a_2)\tau_3(b_1)\tau_3(b_2)}{(a_1^3b_1^2a_2^3b_2^2)^{\frac{1}{2}}} \\
		&\hspace{1in}\times \colk\sum_{c \leq C} \frac{1}{c} \sumstar_{x(c)}   \sum_{\substack{1\leq l_1,l_2\leq 3\\0\leq j_1,j_2\leq 2\\ j_i-l_i=-1}} D_{-l_1}(\eta_1)D_{-l_2}(\eta_2) \sumd_{N,M}\coli_1(c;j_1,j_2).
	\end{aligned}
\end{equation}

As discussed in Section \ref{sec:Outline}, we would like to estimate the factors of $D_{-l_i}$ using Lemma \ref{lem:Dbound} below. However, we will lose too much in our upper bound if we ignore the oscillations present in the integrals $\coli_1(c;j_1,j_2)$. To take advantage of these oscillations, we apply several Fourier-analytic manipulations to suitably transform the sum over $c$ in (\ref{eq:ResidueContribution}). This requires several technical lemmas which we collect in the following subsection. The manipulations themselves are then performed in Subsection \ref{subsec:Fourier}. Finally, we conclude the proof of Proposition \ref{prop:REstimate} in Subsection \ref{subsec:RemainingSeries} by applying Lemma \ref{lem:Dbound} and estimating the remaining Dirichlet series via elementary means.

\subsection{Preliminary Lemmas}\label{subsec:Prelims}

The first three of these are Lemmas 7.1--7.3 of \cite{CL}. 

\begin{lem}\label{lem:sumstar}
	Let $(a,\ell)=1$. We have
	\[
	\sumstar_{\substack{x (c\ell)\\x\equiv a(\ell)}} 1 = c \prod_{\substack{p\mid c\\ p\nmid \ell}} \pth{1-\frac{1}{p}},
	\]
	where the sum is over $x$ coprime to $c\ell$.
\end{lem}

\begin{lem}\label{lem:BesselSplit}
	Let $\alpha,\beta,y_1,y_2$ be nonnegative real numbers satisfying $\alpha y_1,\beta y_2\ll q^2$ and define
	\[
	T(y_1,y_2,\alpha,\beta) = \sum_{\delta=1}^\infty \frac{1}{\delta} J_{k-1} \pth{\frac{4\pi}{\delta} \sqrt{\alpha\beta y_1y_2}} \colk e\pth{\frac{\alpha y_1}{\delta} + \frac{\beta y_2}{\delta}}.
	\]
	Further, let $L = q^{100}$ and $w$ be a smooth function on $\R_{\geq 0}$ with $w(x) = 1$ if $0 \leq x \leq 1$ and $w(x) = 0$ if $x > 2$. Then for any $A > 0$, we have
	\[
	\begin{aligned}
		T &= 2\pi \sum_{\ell=1}^\infty w\fracp{\ell}{L} J_{k-1} \pth{4\pi \sqrt{\alpha y_1\ell}} J_{k-1} \pth{4\pi \sqrt{\beta y_2\ell}} \\
		&\hspace{.25in}-  2\pi \int_{0}^\infty w\fracp{\ell}{L} J_{k-1} \pth{4\pi \sqrt{\alpha y_1\ell}} J_{k-1} \pth{4\pi \sqrt{\beta y_2\ell}} d\ell + O(q^{-2022}).
	\end{aligned}
	\]
\end{lem}

\begin{lem}\label{lem:ZetaIdentity}
	Let $w$ and $L$ be as in Lemma \ref{lem:BesselSplit} and let $u$ be a complex number with $\abs{\Re u} \ll \frac{1}{\log q}$. Then
	\[
	\sum_{\ell=1}^\infty w\fracp{\ell}{L} \frac{1}{\ell^{1+u}}- \int_{0}^\infty w\fracp{\ell}{L}\frac{1}{\ell^{1+u}}\ d\ell = \zeta(1+u) + O(q^{-20}).
	\]
\end{lem}

It can be seen from the proof of this last lemma that the error term is holomorphic in $u$, and thus we can differentiate in $u$ to see that
\begin{equation}\label{eq:ZetaIdentity}
	\sum_{\ell=1}^\infty w\fracp{\ell}{L} \frac{(\log\ell)^j}{\ell^{1+u}}- \int_{0}^\infty w\fracp{\ell}{L}\frac{(\log \ell)^j}{\ell^{1+u}}\ d\ell = (-1)^j\zeta^{(j)}(1+u) + O(q^{-20}).
\end{equation}

The last lemma we need is an improvement on the bound $D_{-i}(\eta) \ll \eta^{-1+\ep}$ used by Djankovi\'{c}.

\begin{lem}\label{lem:Dbound}
	For $i=1,2,3$, we have
	\[
	D_{-i}(\eta) \ll \frac{\tau_2(\eta)(\log \eta)^{3-i}}{\eta}.
	\]
\end{lem}

\begin{proof}
	Note that $D_{-3}(\eta)$ is multiplicative, so
	\[
	D_{-3}(\eta) = \prod_{p^r \mid\mid \eta} D_{-3}(p^r).
	\]
	At a prime power, we have
	\[
	\begin{aligned}
		D_{-3}(p^r) &= \frac{1}{p^{2r}}\sum_{j=0}^r \sum_{\substack{\alpha_1=1\\p^j\mid\mid \alpha_1}}^{p^r} \sum_{\alpha_2=1}^{p^r} \boldsymbol{1}(p^r\mid \alpha_1\alpha_2) =\frac{1}{p^{2r}}\sum_{j=0}^r \sum_{\substack{\alpha_1=1\\p\nmid \alpha_1}}^{p^{r-j}} \sum_{\alpha_2=1}^{p^r} \boldsymbol{1}(p^{r-j}\mid\alpha_2) \\
		&=\frac{1}{p^{2r}}\sum_{j=0}^r \phi(p^{r-j})p^j = \frac{1}{p^{r}} \pth{1+r\pth{1-\frac{1}{p}}} = \frac{r+1}{p^r} \pth{1-\frac{r}{p(r+1)}},
	\end{aligned}
	\]
	and so
	\[
	D_{-3}(\eta) = \frac{\tau_2(\eta)}{\eta} \prod_{p^r\mid\mid\eta} \pth{1-\frac{r}{p(r+1)}} \ll \frac{\tau_2(\eta)}{\eta}.
	\]
	To estimate $D_{-2}$ and $D_{-1}$, we need the following result of Berndt (see \citep{Berndt}, Theorem 2) for the generalized Stieltjes constants (\ref{eq:StieltjesConstantsDef}): for $x\in\hol{0,1}$ and $n\geq 1$, we have
	\[
	\gamma_n(x) \ll_n \frac{\abs{\log x}^n}{x}.
	\]
	Combining this with our estimate for $D_{-3}$, we see that 
	\[
	D_{-2}(\eta) \ll \frac{1}{\eta}\sum_{\alpha_1=1}^\eta \sum_{\alpha_2=1}^\eta \frac{\boldsymbol{1}(\eta | \alpha_1\alpha_2)}{\alpha_1} + \frac{\tau_2(\eta)\log\eta}{\eta}.
	\]
	We then have
	\[
	\begin{aligned}
		\sum_{\alpha_1=1}^\eta \sum_{\alpha_2=1}^\eta \frac{\boldsymbol{1}(\eta | \alpha_1\alpha_2)}{\alpha_1}  &= \sum_{d\mid \eta} \sum_{\substack{\alpha_1=1\\ (\alpha_1,\eta)=\frac{\eta}{d}}}^\eta \sum_{\alpha_2=1}^\eta \frac{\boldsymbol{1}(\eta | \alpha_1\alpha_2)}{\alpha_1}  \\
		&= \sum_{d\mid \eta} \frac{d}{\eta}\sum_{\substack{\alpha_1=1\\ (\alpha_1,d)=1}}^{d} \frac{1}{\alpha_1}\sum_{\alpha_2=1}^\eta \boldsymbol{1}(d | \alpha_2)  =\sum_{d\mid \eta}\sum_{\substack{\alpha_1=1\\ (\alpha_1,d)=1}}^{d} \frac{1}{\alpha_1} \ll \tau_2(\eta)\log \eta,
	\end{aligned}
	\]
	which gives the claimed estimate for $D_{-2}$. Finally, for $D_{-1}$, our arguments above give
	\[
	\begin{aligned}
		D_{-1}(\eta) &\ll \frac{1}{\eta^2}\sum_{\alpha_1=1}^\eta \sum_{\alpha_2=1}^\eta \boldsymbol{1}(\eta | \alpha_1\alpha_2) \pth{(\log \eta)^2 + \frac{\eta\log\eta}{\alpha_1} + \frac{\eta^2}{\alpha_1\alpha_1} + \frac{\eta}{\alpha_1}\log\fracp{\eta}{\alpha_1}} \\
		& \ll \frac{\tau_2(\eta)(\log\eta)^2}{\eta} + \sum_{\alpha_1=1}^\eta\sum_{\alpha_2=1}^\eta \frac{\boldsymbol{1}(\eta | \alpha_1\alpha_2)}{\alpha_1\alpha_2}.
	\end{aligned}
	\]
	Following our treatment of $D_{-2}$, we have
	\[
	\begin{aligned}
		\sum_{\alpha_1=1}^\eta \sum_{\alpha_2=1}^\eta \frac{\boldsymbol{1}(\eta | \alpha_1\alpha_2)}{\alpha_1\alpha_2} &= \sum_{d\mid \eta} \sum_{\substack{\alpha_1=1\\ (\alpha_1,\eta)=\frac{\eta}{d}}}^\eta \sum_{\alpha_2=1}^\eta \frac{\boldsymbol{1}(\eta | \alpha_1\alpha_2)}{\alpha_1\alpha_2} = \sum_{d\mid \eta} \frac{d}{\eta}\sum_{\substack{\alpha_1=1\\ (\alpha_1,d)=1}}^{d} \frac{1}{\alpha_1}\sum_{\alpha_2=1}^\eta \frac{\boldsymbol{1}(d | \alpha_2)}{\alpha_2}\\
		& = \frac{1}{\eta}\sum_{d\mid \eta} \sum_{\substack{\alpha_1=1\\ (\alpha_1,d)=1}}^{d} \frac{1}{\alpha_1}\sum_{a=1}^{\eta/d} \frac{1}{a}  \ll \frac{\tau_2(\eta)(\log\eta)^2}{\eta},
	\end{aligned}
	\]
	which gives the claimed estimate for $D_{-3}$.
	
\end{proof}

\subsection{Fourier-Analytic Manipulations}\label{subsec:Fourier}

Returning to (\ref{eq:ResidueContribution}), we first transform the sums over $N,M,c,x$. This will allow us to determine the main term (in terms of $\log q$) of the double residue $\colr(q)$, which we then estimate by purely arithmetic means in Subsection \ref{subsec:RemainingSeries}.

By the decay of the Bessel function, we may extend the sum over $c \leq C$ to all $c \geq 1$ in a similar way as our truncation in Section \ref{sec:Truncations}. Note that after extending the sum, the only parts of $\colr(q)$ that depend on $N,M$ are the factors of $f$. Because of the absolute convergence of the integrals $\coli_1(c;j_1,j_2)$, we may execute the dyadic sums over $N,M$ to see that
\[
\begin{aligned}
	\sumd_{N,M} \coli_1(c;j_1,j_2)  &= \int_0^\infty \int_0^\infty \frac{(\log y_1)^{j_1}(\log y_2)^{j_2}}{(y_1y_2)^\frac{1}{2}}  U\fracp{a_1^3b_1^2y_1}{q^\frac{3}{2}}U\fracp{a_2^3b_2^2y_2}{q^\frac{3}{2}} \\
	&\hspace{1in}\times e\pth{\frac{a_1^2b_1y_1}{cqa_2b_2}+\frac{a_2^2b_2y_2}{cqa_1b_1}}  J_{k-1}\pth{\frac{4\pi}{cq}\sqrt{a_1a_2y_1y_2}}  dy_1\ dy_2.
\end{aligned}
\]
Thus $\colr(q)$ is
\begin{equation}\label{eq:ResLeft}
	\begin{aligned}
		&\frac{2\pi}{q} \sumfour_{\substack{a_1,b_1,a_2,b_2\geq 1\\ (a_1a_2b_1b_2,q)=1\\ a_i^3b_i^2 \ll q^{3/2+\ep}}} \frac{\mu(a_1)\mu(a_2)\tau_3(b_1)\tau_3(b_2)}{(a_1^3b_1^2a_2^3b_2^2)^{\frac{1}{2}}} \sum_{\substack{1\leq l_1,l_2\leq 3\\0\leq j_1,j_2\leq 2\\ j_i-l_i=-1}} \coli_2(l_1,l_2,j_1,j_2)
	\end{aligned}
\end{equation}
where
\[
\coli_2(l_1,l_2,j_1,j_2) = \int_0^\infty \int_0^\infty \frac{(\log y_1)^{j_1}(\log y_2)^{j_2}}{(y_1y_2)^\frac{1}{2}}  U\fracp{a_1^3b_1^2y_1}{q^\frac{3}{2}}U\fracp{a_2^3b_2^2y_2}{q^\frac{3}{2}} \colc(l_1,l_2;y_1,y_2)  dy_1\ dy_2,
\]
and
\[
	\colc(l_1,l_2;y_1,y_2) =\colk\sum_{c \geq 1} \frac{\colf(c)}{c} \sumstar_{x(c)} D_{-l_1}(\eta_1)D_{-l_2}(\eta_2),
\]
where
\[
\colf(c) = J_{k-1}\fracp{4\pi\sqrt{a_1a_2y_1y_2}}{cq} e\pth{\frac{a_1^2b_1y_1}{cqa_2b_2}+\frac{a_2^2b_2y_2}{cqa_1b_1}}
\]
and $\eta_1,\eta_2$ are as in (\ref{eq:EtaLambdaDef}). We write
\begin{equation}\label{eq:GCDs}
	(a_1 b_1,a_2b_2) = \lambda,\qquad a_1b_1 = u_1\lambda,\qquad a_2b_2 = u_2\lambda,\qquad (u_1 \bar{x}-u_2,c) = (u_2 x-u_1,c) = \delta,
\end{equation}
where $(u_1,u_2)=1$, and so
\begin{equation}\label{eq:etasSimple}
	\eta_1 = \frac{u_2 c/\delta}{(a_1,u_2 c/\delta)}, \qquad \eta_2  =  \frac{u_1 c/\delta}{(a_2,u_1 c/\delta)}.
\end{equation}
We now focus on estimating the integrals $\coli_2$ in (\ref{eq:ResLeft}). To do so, we first transform the cums $\colc$. We proceed by fixing the value of $\delta$ in the sum over $x$ and writing
\begin{equation}\label{eq:CSum1}
	\begin{aligned}
		\colc(l_1,l_2;y_1,y_2) &= \colk\sum_{c\geq 1} \sum_{\delta \mid c} \frac{\colf(c)}{c}\sumstar_{\substack{x(c)\\ (u_2 x-u_1,c)=\delta}}  D_{-l_1}\fracp{u_2 c/\delta}{(a_1,u_2 c/\delta)}D_{-l_2}\fracp{u_1 c/\delta}{(a_2,u_1 c/\delta)} \\
		&= \colk\sum_{\delta \geq 1} \frac{1}{\delta} \sum_{c \geq 1}\frac{\colf(c\delta)}{c} D_{-l_1}\fracp{u_2 c}{(a_1,u_2 c)}D_{-l_2}\fracp{u_1 c}{(a_2,u_1 c)}  \sumstar_{\substack{x(c\delta )\\ (\frac{u_2 x-u_1}{\delta},c)=1}} 1.
	\end{aligned}
\end{equation}
M\"{o}bius inversion in the sum over $x$ gives
\[
\sumstar_{\substack{x(c\delta )\\ (\frac{u_2 x-u_1}{\delta},c)=1}} 1 =\sum_{b\mid c} \mu(b) \sumstar_{\substack{x(c\delta )\\ u_2 x \equiv u_1 (b\delta)}} 1.
\]
Note that since $(x,b\delta)=(u_1,u_2)=1$, the congruence $u_2 x \equiv u_1 (b\delta)$ has a solution in $x$ if and only if $(u_1u_2,b\delta)=1$. Applying Lemma \ref{lem:sumstar}, the last line of (\ref{eq:CSum1}) can be written
\[
\sum_{c\geq 1} c \sum_{\substack{b \geq 1\\ (u_1u_2,b)=1}} \frac{\mu(b)}{b} D_{-l_1}\fracp{u_2 cb}{(a_1,u_2 cb)}D_{-l_2}\fracp{u_1 cb}{(a_2,u_1 cb)} \colk\sum_{\substack{\delta \geq 1\\ (u_1u_2,\delta)=1}} \frac{\colf(cb\delta)}{\delta} \prod_{\substack{p\mid c\\p\nmid b\delta}} \pth{1-\frac{1}{p}}.
\]
Several more applications of M\"{o}bius inversion give
\[
\begin{aligned}
	\sum_{\substack{\delta \geq 1\\ (u_1u_2,\delta)=1}} \frac{\colf(cb\delta)}{\delta} \prod_{p\mid (c,b\delta h)} \pth{1-\frac{1}{p}} 
	&= \sum_{h\mid u_1u_2} \frac{\mu(h)}{h} \sum_{\gamma\mid c} \frac{1}{\gamma} \sum_{g\mid \frac{c}{\gamma}} \frac{\mu(g)}{g} \prod_{\substack{p\mid c\\p\nmid b\gamma h}} \pth{1-\frac{1}{p}} \sum_{\delta \geq 1} \frac{\colf(cbhg\gamma\delta)}{\delta},
\end{aligned}
\]
and thus
\[
\begin{aligned}
	\colc(l_1,l_2;y_1,y_2) &= \sum_{h\mid u_1u_2} \frac{\mu(h)}{h} \sum_{\substack{b \geq 1\\ (u_1u_2,b)=1}} \frac{\mu(b)}{b} \sum_{c\geq 1} D_{-l_1}\fracp{u_2 cb}{(a_1,u_2 cb)}D_{-l_2}\fracp{u_1 cb}{(a_2,u_1 cb)} \\
	&\times \sum_{\gamma\mid c} \frac{1}{\gamma} \sum_{g\mid \frac{c}{\gamma}} \frac{\mu(g)}{g} \prod_{\substack{p\mid c\\p\nmid b\gamma h}} \pth{1-\frac{1}{p}} \colk\sum_{\delta \geq 1} \frac{\colf(cbhg\gamma\delta)}{\delta}.
\end{aligned}
\]
We now apply Lemma \ref{lem:BesselSplit} with
\[
\alpha = \frac{a_1^2b_1}{a_2b_2cbhg\gamma\delta q}, \qquad \beta = \frac{a_2^2b_2}{a_1b_1cbhg\gamma\delta q}.
\]
Note that $\alpha y_1,\beta y_2\ll q^{\frac{3}{2}+\ep}$ by the decay of $U$. Thus with negligible error, we have
\[
\begin{aligned}
	\colk\sum_{\delta \geq 1} &\frac{\colf(cbhg\gamma\delta)}{\delta} \\
	&= 2\pi \sumpth{\sum_{\ell \geq 1}w\fracp{\ell}{L} - \int_{0}^{\infty}w\fracp{\ell}{L} d\ell} J_{k-1}\pth{4\pi \sqrt{\frac{a_1^2b_1y_1\ell}{a_2b_2cbhg\gamma q}}} J_{k-1}\pth{4\pi \sqrt{\frac{a_2^2b_2y_2\ell}{a_1b_1cbhg\gamma q}}}.
\end{aligned}
\]
For brevity, we set
\[
\cola_1 = 4\pi\sqrt{\frac{a_1^2b_1}{a_2b_2cbhg\gamma q}}, \qquad \cola_2 = 4\pi\sqrt{\frac{a_2^2b_2}{a_1b_1cbhg\gamma q}}.
\]
Returning to the definition of $\coli_2$, the above analysis and change of variables $y_i\to y_i^2$ give
\begin{equation}\label{eq:ResInAll}
	\begin{aligned}
		\coli_2(l_1,l_2,j_1,j_2) &= Q  \sum_{h\mid u_1u_2} \frac{\mu(h)}{h} \sum_{\substack{b \geq 1\\ (u_1u_2,b)=1}} \frac{\mu(b)}{b} \sum_{c \geq 1}cD_{-l_1}\fracp{u_2 cb}{(a_1,u_2 cb)}D_{-l_2}\fracp{u_1 cb}{(a_2,u_1 cb)} \\
		&\phantom{=}\times \sum_{\gamma\mid c} \frac{1}{\gamma} \sum_{g\mid \frac{c}{\gamma}} \frac{\mu(g)}{g} \prod_{\substack{p\mid c\\p\nmid b\gamma h}} \pth{1-\frac{1}{p}}   \sumpth{\sum_{\ell \geq 1}w\fracp{\ell}{L} - \int_{0}^{\infty}w\fracp{\ell}{L} d\ell} \\
		&\phantom{=}\times\int_0^\infty (\log y_1)^{j_1} U\fracp{a_1^3b_1^2y_1}{q^\frac{3}{2}} J_{k-1}\pth{\cola_1y_1\sqrt{\ell}} dy_1\\
		&\phantom{=}\times \int_0^\infty (\log y_2)^{j_2} U\fracp{a_2^3b_2^2y_2}{q^\frac{3}{2}} J_{k-1}\pth{\cola_2y_2\sqrt{\ell}} dy_2.
	\end{aligned}
\end{equation}
Let $\coli_3(\ell,\cola_1,\cola_2)$ denote the product of integrals on the last line. Here and throughout this section, we let $Q$ denote a positive constant, not necessarily the same at each occurrence, depending at most on $j_1,j_2$. Opening the factor of $U$, the integral in $y_1$ is
\begin{equation}\label{eq:PreHankel}
	\frac{Q}{2\pi i} \int\limits_{(\alpha_1)} \gamma(s_1)^3 G^3(s_1) \fracp{q^\frac{3}{2}}{a_1^3b_1^2}^{s_1} \int_0^\infty (\log y_1)^{j_1} y_1^{-2s_1} J_{k-1}\pth{\cola_1y_1 \sqrt{\ell}} dy_1\frac{ds_1}{s_1},
\end{equation}
where $\alpha_1 > 0$, and a similar expression holds for the integral in $y_2$. The inner integrals are Hankel transforms which can be evaluated explicitly using equation 6.561.14 of \cite{IntegralTable}, which is
\[
\int_{0}^{\infty} x^\mu J_\nu(ax) dx = 2^\mu a^{-\mu-1} \frac{\Gamma(\frac{1}{2}+\frac{1}{2} \nu + \frac{1}{2}\mu)}{\Gamma(\frac{1}{2}+\frac{1}{2} \nu - \frac{1}{2}\mu)}, \qquad (-\Re \nu-1 < \Re \mu < \half,\ a>0).
\]
Differentiating this with respect to $\mu$, we obtain 
\begin{equation}\label{eq:BesselHankel}
	\begin{aligned}
		\int_{0}^{\infty} x^\mu (\log x)^j J_\nu(ax) dx = (-1)^j 2^\mu a^{-\mu-1} \frac{\Gamma(\frac{1}{2}+\frac{1}{2} \nu + \frac{1}{2}\mu)}{\Gamma(\frac{1}{2}+\frac{1}{2} \nu - \frac{1}{2}\mu)} \colp_j(\log a,\mu,\nu),
	\end{aligned}
\end{equation}
for $j\geq 0$, where $\colp_j(w,\mu,\nu)$ is a monic polynomial of degree $j$ in $w$ with coefficients involving polygamma functions and the parameters $\mu,\nu$. For instance,
\[
\colp_1(w) = w - \log 2 - \frac{1}{2} \frac{\Gamma'(\frac{1}{2}+\frac{1}{2} \nu - \frac{1}{2}\mu)}{\Gamma(\frac{1}{2}+\frac{1}{2} \nu - \frac{1}{2}\mu)} - \frac{1}{2} \frac{\Gamma'(\frac{1}{2}+\frac{1}{2} \nu + \frac{1}{2}\mu)}{\Gamma(\frac{1}{2}+\frac{1}{2} \nu + \frac{1}{2}\mu)}.
\]
In the present case, specifying $\nu =k-1$ and $\mu = -2s_i$, the coefficients will be holomorphic and of rapid decay on vertical lines so long as $\alpha_i > -k/2 + \ep$, say. Applying (\ref{eq:BesselHankel}), we see that (\ref{eq:PreHankel}) is
\[
\begin{aligned}
	&\frac{(-1)^{j_1}Q}{2\pi i} \int\limits_{(\alpha_1)} \gamma(s_1)^3 G^3(s_1) \fracp{4\pi^2q^\frac{3}{2}}{a_1^3b_1^2}^{s_1}  \pth{\sqrt{\frac{a_1^2b_1\ell}{a_2b_2 cbhg\gamma q}}}^{2s_1-1}  \frac{\Gamma(\frac{k}{2}-s_1)}{\Gamma(\frac{k}{2}+s_1)} \colp_j\pth{\log\pth{\cola_1\sqrt{\ell}},s_1}\frac{ds_1}{s_1}.
\end{aligned}
\]
Here we have suppressed the dependence of $\colp_j$ on $k$. A similar expression holds for the integral in $y_2$, and thus $\coli_3(\ell,\cola_1,\cola_2)$ is
\[
\begin{aligned}
	&(-1)^{j_1+j_2}Q\frac{cbhg\gamma q}{\ell \sqrt{a_1a_2}}\fracp{1}{2\pi i}^2 \int\limits_{(\alpha_1)} \int\limits_{(\alpha_2)}  \gamma(s_1)^3 G^3(s_1)\gamma(s_2)^3 G^3(s_2) \frac{\Gamma(\frac{k}{2}-s_1)}{\Gamma(\frac{k}{2}+s_1)} \frac{\Gamma(\frac{k}{2}-s_2)}{\Gamma(\frac{k}{2}+s_2)} \\
	&\hspace{.5in}\times  \fracp{4\pi^2q^\frac{1}{2}\ell}{a_1b_1a_2b_2cbhg\gamma}^{s_1+s_2}\colp_{j_1}\pth{\log\pth{\cola_1\sqrt{\ell}},s_1}\colp_{j_2}\pth{\log\pth{\cola_2\sqrt{\ell}},s_2}\frac{ds_2}{s_2}\frac{ds_1}{s_1}.
\end{aligned}
\]
Let $P_{j_i}(n,s_i)$ denote the coefficient of $w^{n}$ in $\colp_{j_i}(w,s_i)$. We view the product of $\colp_{j_1}$ and $\colp_{j_2}$ as a polynomial $\colp$ in $\log \ell$ of degree $j_1+j_2$, where the coefficient of $(\log \ell)^n$ is given by
\[
\begin{aligned}
P_n(\log \cola_1,&\log\cola_2,s_1,s_2) \\
&= \frac{1}{2^n} \sum_{n=k_1+k_2} \sum_{n_1=k_1}^{j_1} \sum_{n_2=k_2}^{j_2} \binom{n_1}{k_1} \binom{n_2}{k_2} P_{j_1}(n_1,s_1)P_{j_2}(n_2,s_2) (\log \cola_1)^{n_1-k_1} (\log \cola_2)^{n_2-k_2}.
\end{aligned}
\]
Applying (\ref{eq:ZetaIdentity}), we see that
\[
\begin{aligned}
\sum_{\ell \geq 1}w\fracp{\ell}{L} \frac{\colp(\log \ell)}{\ell^{1-s_1-s_2}}- \int_{0}^{\infty}w\fracp{\ell}{L}  &\frac{\colp(\log \ell)}{\ell^{1-s_1-s_2}} d\ell \\
&= \sum_{n=0}^{j_1+j_2} P_n(\log \cola_1,\log\cola_2,s_1,s_2) \zeta^{(n)}(1-s_1-s_2) + O(q^{-20}) \\
&=: \colp^*(\log \cola_1,\log \cola_2,s_1,s_2),
\end{aligned}
\]
say. Up to a negligible error term, we then have
\[
\begin{aligned}
	&\sumpth{\sum_{\ell \geq 1}w\fracp{\ell}{L} - \int_{0}^{\infty}w\fracp{\ell}{L} d\ell}\coli_3(\ell,\cola_1,\cola_2) \\
	&\hspace{.5in} = Q\frac{cbhg\gamma q}{\sqrt{a_1a_2}}\fracp{1}{2\pi i}^2 \int\limits_{(\alpha_1)} \int\limits_{(\alpha_2)}  \gamma(s_1)^3 G^3(s_1)\gamma(s_2)^3 G^3(s_2) \frac{\Gamma(\frac{k}{2}-s_1)}{\Gamma(\frac{k}{2}+s_1)} \frac{\Gamma(\frac{k}{2}-s_2)}{\Gamma(\frac{k}{2}+s_2)} \\
	&\hspace{.5in}\times  \fracp{4\pi^2q^\frac{1}{2}}{a_1b_1a_2b_2cbhg\gamma}^{s_1+s_2}\colp^*(\log \cola_1,\log \cola_2,s_1,s_2) \frac{ds_2}{s_2}\frac{ds_1}{s_1}.
\end{aligned}
\]
say. Note the coefficient of $(\log \cola_1)^{j_1} (\log \cola_2)^{j_2}$ in $\colp^*(\log \cola_1,\log \cola_2,s_1,s_2)$ is $\zeta(1-s_1-s_2)$. We deal only with the contribution of this term, as it will be clear from our analysis that the other terms of $\colp^*$ can be treated similarly. Returning to (\ref{eq:ResInAll}), we see that the representative term of $\coli_2(l_1,l_2,j_1,j_2)$ is
\begin{equation}\label{eq:RepresentativeOfRes}
	\begin{aligned}
		& Q\frac{q}{\sqrt{a_1a_2}}  \sum_{h\mid u_1u_2} \mu(h) \sum_{\substack{b \geq 1\\ (u_1u_2,b)=1}} \mu(b) \sum_{c \geq 1} c D_{-l_1}\fracp{u_2 cb}{(a_1,u_2 cb)}D_{-l_2}\fracp{u_1 cb}{(a_2,u_1 cb)} \\
		&\phantom{=}\times \sum_{\gamma\mid c} \sum_{g\mid \frac{c}{\gamma}} \mu(g) \prod_{\substack{p\mid c\\p\nmid b\gamma h}} \pth{1-\frac{1}{p}} (\log \cola_1)^{j_1}(\log \cola_2)^{j_2} \coly(cbhg\gamma),
	\end{aligned}
\end{equation}
where
\[
\begin{aligned}
	\coly(w) &= \fracp{1}{2\pi i}^2 \int\limits_{(1)} \int\limits_{(1)}  \gamma(s_1)^3 G^3(s_1)\gamma(s_2)^3 G^3(s_2) \frac{\Gamma(\frac{k}{2}-s_1)}{\Gamma(\frac{k}{2}+s_1)} \frac{\Gamma(\frac{k}{2}-s_2)}{\Gamma(\frac{k}{2}+s_2)} \fracp{4\pi^2q^\frac{1}{2}}{a_1b_1a_2b_2w}^{s_1+s_2} \\
	&\hspace{.5in}\times \zeta(1-s_1-s_2) \frac{ds_2}{s_2}\frac{ds_1}{s_1} \\
	&=\fracp{1}{2\pi i}^2 \int\limits_{(1)} \int\limits_{(1)} \colg(s_1,s_2) \fracp{4\pi^2q^\frac{1}{2}}{a_1b_1a_2b_2w}^{s_1+s_2} \zeta(1-s_1-s_2) \frac{ds_2}{s_2}\frac{ds_1}{s_1},
\end{aligned}
\]
say, and we have taken the lines of integration to 1. Here $\colg$ is holomorphic and decays rapidly on vertical lines so long as $\Re s_1,\Re s_1 < \frac{k}{2}$. Changing variables, we obtain
\begin{equation}\label{eq:Y(w)def}
	\coly(w) =\fracp{1}{2\pi i}^2 \int\limits_{(1)} \int\limits_{(2)} \colg(s,z-s) \fracp{4\pi^2q^\frac{1}{2}}{a_1b_1a_2b_2w}^{z} \zeta(1-z) \frac{du}{z-s}\frac{ds}{s}.
\end{equation}

We deal first with the case $cbhg\gamma > q$. For $w > q$, we take the line of integration in $z$ to $\Re u = \frac{9}{2}$. Applying the functional equation for $\zeta$ and Stirling's formula, we see that in this case
\[
\coly(w) \ll q^{\frac{9}{4}}\frac{1}{(a_1b_1a_2b_2w)^{\frac{9}{2}}}.
\]
Recalling the definition of $\cola_1,\cola_2$ and noting the ranges of summations for the variables in the expression $\coli_2$, we have
\[
\log \cola_1,\log \cola_2 \ll \log(bcq) \ll (cbq)^\ep.
\]
Using the trivial bound $D_{-i}(\eta) \ll 1$, we find that
\[
\begin{aligned}
	&\sum_{c > q} c D_{-l_1}\fracp{u_2 cb}{(a_1,u_2 cb)}D_{-l_2}\fracp{u_1 cb}{(a_2,u_1 cb)} \sum_{\gamma\mid c} \sum_{g\mid \frac{c}{\gamma}} \mu(g) \prod_{\substack{p\mid c \\ p\nmid b\gamma h}} \pth{1-\frac{1}{p}} \coly(j_1,j_2,cbhg\gamma) \\
	&\ll \frac{q^{\frac{9}{4}+\ep} b^\ep}{(a_1b_1a_2b_2bh)^{\frac{9}{2}}} \sum_{c > q} c^{-\frac{7}{2}+\ep} \ll q^{-\frac{1}{4}+\ep}\frac{ b^\ep}{(a_1b_1a_2b_2bh)^{\frac{9}{2}}}.
\end{aligned}
\]
Using (\ref{eq:ResLeft}) and (\ref{eq:ResInAll}), we see that the contribution to $\colr(q)$ from those terms with $c > q$ is $O(q^{-\frac{1}{4}+\ep})$ (the error term can be improved here, but this suffices for our purposes). 

\subsection{The Remaining Dirichlet Series}\label{subsec:RemainingSeries}

Let $\colh(c)$ denote the product of $c$ and the factors of $D_{-l_i}$. We consider the sum over $c$,
\[
\sum_{c \leq q} \colh(c) \sum_{\gamma\mid c} \sum_{g\mid \frac{c}{\gamma}} \mu(g) \prod_{\substack{p\mid c\\p\nmid \gamma bh}} \pth{1-\frac{1}{p}} (\log \cola_1)^{j_1}(\log \cola_2)^{j_2} \coly(cbhg\gamma).
\]
Since $b$ and $h$ are squarefree, we may rewrite this as
\[
\sum_{d\mid bh} \sum_{\substack{\lambda c_d \geq 1\\ (\lambda,bh)=1\\ p\mid c_d\iff p\mid d}} \colh(\lambda c_d) \sum_{\gamma\mid \lambda c_d} \prod_{\substack{p\mid \lambda c_d\\p\nmid \gamma bh}} \pth{1-\frac{1}{p}} \sum_{g\mid \frac{\lambda c_d}{\gamma}} \mu(g)  (\log \cola_1)^{j_1}(\log \cola_2)^{j_2} \coly(\lambda c_d bhg\gamma),
\]
and the variable $c$ has been modified in the definitions of $\cola_1,\cola_2$. For fixed $d\mid bh$ and $\gamma\mid \lambda c_d$, we write $\gamma=\gamma_1\gamma_2$, where $\gamma_1\mid \lambda$ and $\gamma_2\mid c_d$ so that
\[
\prod_{\substack{p\mid \lambda c_d\\p\nmid \gamma bh}} \pth{1-\frac{1}{p}} = \frac{\phi(\lambda c_d)}{\lambda c_d} \prod_{p\mid (\lambda c_d,\gamma bh)} \pth{1-\frac{1}{p}}^{-1} = \frac{\phi(\lambda c_d)}{\lambda c_d} \prod_{p\mid \gamma_1} \pth{1-\frac{1}{p}}^{-1} \prod_{p\mid c_d} \pth{1-\frac{1}{p}}^{-1} = \frac{\gamma_1\phi(\lambda)}{\lambda\phi(\gamma_1)}.
\]
The sum over $\gamma$ becomes
\[
\begin{aligned}
	&\frac{\phi(\lambda)}{\lambda} \sum_{\gamma_1\mid \lambda} \frac{\gamma_1}{\phi(\gamma_1)} \sum_{g_1\mid \frac{\lambda}{\gamma_1}} \mu(g) \sum_{\gamma_2\mid c_d} \sum_{g_2\mid \frac{c_d}{\gamma_2}} \mu(g_2)  (\log \cola_1)^{j_1}(\log \cola_2)^{j_2} \coly(\lambda c_d bhg_1g_2\gamma_1\gamma_2),
\end{aligned}
\]
where again, the variables in the definitions of $\cola_1,\cola_2$ have been appropriately modified. Moving the sums over $\gamma_1,\gamma_2$ inside the integral, we are led to consider the functions
\begin{equation}
	\begin{aligned}
		\mathscr{C}_1(n) = \mathscr{C}_1(n,z;j_1,j_2) &= \frac{\phi(n)}{n}\sum_{\gamma\mid n} \frac{\gamma}{\phi(\gamma)} \frac{(\log \gamma)^{j_1}}{\gamma^z} \sum_{g\mid \frac{n}{\gamma}} \frac{\mu(g)(\log g)^{j_2}}{g^z}, \\
		\mathscr{C}_2(n) = \mathscr{C}_2(n,z;j_1,j_2) &= \sum_{\gamma\mid n} \frac{(\log \gamma)^{j_1}}{\gamma^z} \sum_{g\mid \frac{n}{\gamma}} \frac{\mu(g)(\log g)^{j_2}}{g^z},
	\end{aligned}
\end{equation}
where now the $j_1,j_2$ are arbitrary nonnegative integers.
\begin{prop}\label{prop:FinalPiece}
	For all integers $n\geq 2$, $j_1,j_2\geq 0$, and $z$ with $\Re z \geq 0$, we have
	\[
	\mathscr{C}_1(n),\mathscr{C}_2(n) \leq (\log n)^{j_1+j_2}.
	\]
\end{prop}

\begin{proof}
	Let $v(n)$ be either $1$ or $\frac{\phi(n)}{n}$ and consider the function
	\[
	\mathscr{C}(n,z,s) = v(n)\sum_{\gamma\mid n} \frac{1}{v(\gamma)\gamma^z} \sum_{g\mid \frac{n}{\gamma}} \frac{\mu(g)}{g^s}.
	\]
	We will show that
	\[
	\left[\frac{\partial^{j_1}}{\partial z^{j_1}} \frac{\partial^{j_2}}{\partial s^{j_2}} \mathscr{C}(n,z,s)\right]_{s=u} \leq (\log n)^{j_1+j_2}
	\]
	for all $n\geq 2$ and $z$ with $\Re z \geq 0$. We have
	\[
	\begin{aligned}
		\mathscr{C}(n,z,s) &= \frac{v(n)}{n^u} \prod_{p^r\mid\mid n} \pth{\frac{1}{v(p^r)} + \pth{1-\frac{1}{p^s}} \pth{\frac{p^z}{v(p^{r-1})} + \cdots +\frac{p^{(r-1)z}}{v(p)} + p^{ru}}} \\
		&= \prod_{p^r\mid\mid n} \pth{\frac{1}{p^{rz}}+ \pth{1-\frac{1}{p^s}} \pth{v(p) + \frac{1}{p^{rz}}\fracp{p^{rz} - p^z}{p^z-1} }} \\
		&= \prod_{p^r\mid\mid n} \mathscr{C}_p(z,s),
	\end{aligned}
	\]
	say, since $v(p^r) = v(p)$ for all $p$ and $r\geq 1$. For $s=z$, we have
	\[
	\mathscr{C}(n,z,z) = \prod_{p^r\mid\mid n} \pth{v(p)+\frac{v(p)-1}{p^z}},
	\]
	and specifying $v(p) = 1$ and $v(p) = 1-\frac{1}{p}$, it follows that $\abs{\mathscr{C}(n,z,z)} \leq 1$. This gives the case $j_1=j_2=0$.
	
	To produce the logarithmic factors, we differentiate in $z$ and $s$. Writing $n=p_1^{r_1}\cdots p_{\omega(n)}^{r_{\omega(n)}}$, we have
	\[
	\frac{\partial^j}{\partial s^j} \mathscr{C}(n,z,s) = \sum_{j_1+\cdots+j_{\omega(n)}=j} \binom{j}{j_1,\ldots,j_{\omega(n)}} \prod_{i=1}^{\omega(n)} \frac{\partial^{j_i}}{\partial s^{j_i}} \mathscr{C}_p(z,s),
	\]
	where $\binom{j}{j_1,\ldots,j_{\omega(n)}}$ denotes the multinomial coefficient. A similar expression holds for the $j$th partial derivative with respect to $z$. For $j\geq 1$, we have
	\[
	\begin{aligned}
		\frac{\partial^{j}}{\partial s^{j}} \mathscr{C}_p(z,s) &= (-1)^{j-1} \fracp{(\log p)^j}{p^s} \pth{v(p) + \frac{1}{p^{rz}}\fracp{p^{rz} - p^z}{p^z-1} },\\
		\frac{\partial^{j}}{\partial z^{j}} \mathscr{C}_p(z,s) &= (-\log p)^{j} \pth{\frac{r^j}{p^{rz}}+ \pth{1-\frac{1}{p^s}}\pth{\frac{1}{p^z} + \frac{2^j}{p^{2z}} + \cdots + \frac{(r-1)^j}{p^{(r-1)z}}}}.
	\end{aligned}
	\]
	Since $v(p) \leq 1$, we have
	\[
	\abs{v(p) + \frac{1}{p^{rz}}\fracp{p^{rz} - p^z}{p^z-1} } \leq r,
	\]
	and thus
	\[
	\abs{\frac{\partial^{j}}{\partial s^{j}} \mathscr{C}_p(z,s)} \leq r(\log p)^j
	\]
	for $\Re z,\Re s \geq 0$. Likewise, if we set $s=z$, then
	\[
	\left[\frac{\partial^{j}}{\partial z^{j}} \mathscr{C}_p(z,s)\right]_{s=u} = (-\log p)^{j} \pth{\frac{1}{p^{rz}}+ \frac{2^j-1}{p^{2z}} + \cdots \frac{r^j-(r-1)^j}{p^{rz}}},
	\]
	and thus 
	\[
	\abs{\left[\frac{\partial^{j}}{\partial z^{j}} \mathscr{C}_p(u,s)\right]_{s=u}} \leq (\log p)^j (1+(2^j-1) + \cdots + (r^j-(r-1)^j)) = (r\log p)^j
	\]
	so long as $\Re z \geq 0$. Since we have already shown that $\abs{\mathscr{C}_p(z,s)} \leq 1$, we deduce that
	\[
	\abs{\left[\frac{\partial^j}{\partial s^j} \mathscr{C}(n,z,s)\right]_{s=z}} \leq \sum_{j_1+\cdots+j_{\omega(n)}=j} \binom{j}{j_1,\ldots,j_{\omega(n)}} \prod_{\substack{i=1\\ j_i>0}}^{\omega(n)} r_i(\log p_i)^{j_i} \leq (\log n)^j,
	\]
	and the same estimate holds for the $j$th partial with respect to $u$ evaluated when $s=u$. It remains to deal with the case when both $j_1,j_2$ are nonzero. In this case, we have
	\[
	\frac{\partial^{j_1}}{\partial u^{j_1}}\frac{\partial^{j_2}}{\partial s^{j_2}} \mathscr{C}_p(z,s) = \frac{-(-\log p)^{j_1+j_2}}{p^s} \pth{\frac{1}{p^z} + \frac{2^{j_1}}{p^{2z}} + \cdots + \frac{(r-1)^{j_1}}{p^{(r-1)z}}},
	\]
	and as before, we find that
	\[
	\begin{aligned}
		\abs{\left[\frac{\partial^{j_1}}{\partial z^{j_1}} \mathscr{C}(n,z,s)\right]_{s=z}} &= \sumabs{\sum_{\substack{k_1+\cdots+k_{\omega(n)}=j_1\\ l_1+\cdots+l_{\omega(n)}=j_2}} \binom{j_1}{k_1,\ldots,k_{\omega(n)}}\binom{j_2}{l_1,\ldots,l_{\omega(n)}} \prod_{i=1}^{\omega(n)} \frac{\partial^{k_i}}{\partial z^{k_i}}\frac{\partial^{l_i}}{\partial s^{l_i}} \mathscr{C}_p(z,s)} \\
		&\leq \sum_{\substack{k_1+\cdots+k_{\omega(n)}=j_1\\ l_1+\cdots+l_{\omega(n)}=j_2}} \binom{j_1}{k_1,\ldots,k_{\omega(n)}}\binom{j_2}{l_1,\ldots,l_{\omega(n)}} \prod_{\substack{i=1\\ (k_i,l_i) \neq (0,0)}}^{\omega(n)} (r_i\log p_i)^{k_i+l_i} \\
		&\leq (\log n)^{j_1+j_2}.
	\end{aligned}
	\]
\end{proof}

Returning to our analysis, we now study the sum over $c\leq q$ with the additional assumption that $cbhg\gamma \leq q$, so $\log \cola_1,\log \cola_2\ll \log q$. We decompose the sum over $c$ as in the beginning of this subsection, move the sums over $g,\gamma$ inside the integral, take the line of integration in $\coly$ to $\Re z =\frac{1}{\log q}$, and apply Proposition \ref{prop:FinalPiece}. The factor $(\log \cola_1)^{j_1}(\log \cola_2)^{j_2}$ in (\ref{eq:RepresentativeOfRes}) produces products of logarithms of various combinations of the summation variables, but no matter how they are arranged, their boundedness by $\log q$ combined with Proposition \ref{prop:FinalPiece} shows that we obtain a power of $(\log q)^{j_1+j_2}$ that may be factored through the entire sum after applying the triangle inequality. From (\ref{eq:ResLeft}) and (\ref{eq:RepresentativeOfRes}), we deduce that the representative term of $\colr(q)$ is bounded by
\[
\begin{aligned}
	&\sumfour_{\substack{a_1,b_1,a_2,b_2\geq 1\\ (a_1a_2b_1b_2,q)=1\\ a_i^3b_i^2 \ll q^{3/2+\ep}}} \frac{\tau_3(b_1)\tau_3(b_2)\tau_2(u_1u_2)}{a_1^2b_1a_2^2b_2} \sum_{\substack{1\leq l_1,l_2\leq 3\\0\leq j_1,j_2\leq 2\\ j_i-l_i=-1}} (\log q)^{j_1+j_2+1} \\
	&\hspace{.25in}\times\sum_{b \geq 1}  \sum_{c \leq q} c \abs{D_{-l_1}\fracp{u_2 cb}{(a_1,u_2 cb)}D_{-l_2}\fracp{u_1 cb}{(a_2,u_1 cb)}},
\end{aligned}
\]
where the extra $\log q$ comes from the factor of $\zeta$ in $\coly$ and we have ignored the contribution from $c>q$. Each variable in the summation is bounded a power of $q$, and thus so are the arguments of  $D_{-l_i}$. We note at this point that if one considers a term other than the leading term in $\colp^*$, one obtains a higher power of $\log q$ from the zeta factor (which will have been differentiated some additional number of times), but the total powers of logarithms of the other variables are smaller, and so we still obtain the same estimate as above. Thus after applying Lemma \ref{lem:Dbound}, we obtain
\[
\begin{aligned}
	\colr(q) &\ll (\log q)^5\sumfour_{\substack{a_1,b_1,a_2,b_2\geq 1\\ (a_1a_2b_1b_2,q)=1\\ a_i^3b_i^2 \ll q^{3/2+\ep}}} \frac{\tau_3(b_1)\tau_3(b_2)\tau_2(u_1u_2)}{a_1^2b_1a_2^2b_2} \\
	&\hspace{.25in}\times\sum_{b \geq 1}  \frac{1}{b^2}\sum_{c \leq q} \frac{(a_1,u_2cb)(a_2,u_1cb)}{c} \tau_2\fracp{u_2 cb}{(a_1,u_2 cb)}\tau_2\fracp{u_1 cb}{(a_2,u_1 cb)}.
\end{aligned}
\]
Using the bounds 
\[
\tau_j(ab) \leq \tau_j(a)\tau_j(b), \qquad \tau_j\fracp{a}{d} \leq \tau_j(a)\ \text{if $d\mid a$}, \qquad \tau_j(a) \ll a^\ep,\qquad (a,bc) \leq (a,b)(a,c),
\]
and neglecting several summation conditions, we find that
\[
\begin{aligned}
	\colr(q)&\ll(\log q)^5 \sumfour_{a_1,b_1,a_2,b_2\geq 1} \frac{(a_1b_1a_2b_2)^\ep (a_1b_1,a_2b_2)^2(a_1,u_2)(a_2,u_1)}{(a_1a_2)^3(b_1b_2)^2}  \\
	&\hspace{.25in}\times\sumpth{\sum_{b \geq 1} \frac{(a_1a_2,b)\tau_2(b)^2}{b^2}}\sumpth{\sum_{c \leq q} \frac{(a_1,c)(a_2,c)\tau_2(c)^2}{c}}.
\end{aligned}
\]
The sum over $b$ is
\[
\sum_{d\mid a_1a_2} \frac{1}{d} \sum_{\substack{b \geq 1\\ (\frac{a_1a_2}{d},b)=1}} \frac{\tau_2(bd)^2}{b^2}  \leq \sum_{d\mid a_1a_2} \frac{\tau_2(d)^2}{d} \sum_{b \geq 1} \frac{\tau_2(b)^2}{b^2} \ll \tau_2(a_1a_2) \ll (a_1a_2)^\ep,
\]
so we are left to consider
\[
\begin{aligned}
	&&(\log q)^5 \sumfour_{a_1,b_1,a_2,b_2\geq 1} \frac{(a_1b_1a_2b_2)^\ep (a_1b_1,a_2b_2)^2(a_1,u_2)(a_2,u_1)}{(a_1a_2)^3(b_1b_2)^2} \sum_{c \leq q} \frac{(a_1,c)(a_2,c)\tau_2(c)^2}{c}.
\end{aligned}
\]
Let $\delta = (a_1,a_2)$ and write $a_1 = \delta \lambda_1$, $a_2 = \delta \lambda_2$ with $(\lambda_1,\lambda_2)=1$. Then the sum over $c$ is
\[
\begin{aligned}
	&\sum_{d\mid \delta} \sum_{\substack{c \leq q\\ (c,\delta)=d}} \frac{(\delta \lambda_1,c)(\delta \lambda_2,c)\tau_2(c)^2}{c} \leq \sum_{d\mid \delta} \tau_2(d)^2 d \sum_{\substack{c \leq q/d\\ (c,\frac{\delta}{d})=1}} \frac{(\lambda_1\lambda_2,c)\tau_2(c)^2}{c} \\
	&\leq \sum_{d\mid \delta} \tau_2(d)^2 d \sum_{c \leq q} \frac{(\lambda_1\lambda_2,c)\tau_2(c)^2}{c} \leq \sum_{d\mid \delta} \tau_2(d)^2 d \sum_{g\mid \lambda_1\lambda_2}  \sum_{c \leq q} \frac{\tau_2(cg)^2}{c} \\
	&\leq \sum_{d\mid \delta} \tau_2(d)^2 d \sum_{g\mid \lambda_1\lambda_2} \tau_2(g)^2 \sum_{c \leq q} \frac{\tau_2(c)^2}{c}.
\end{aligned}
\]
The inner sum over $c$ on the right is bounded by $(\log q)^4$, the sum over $g$ by $(a_1a_2)^\ep$, and the sum over $d$ by $(a_1a_2)^\ep (a_1,a_2)$. Thus
\[
\colr(q) \ll (\log q)^9 \sumfour_{a_1,b_1,a_2,b_2\geq 1} \frac{(a_1b_1a_2b_2)^\ep (a_1b_1,a_2b_2)^2(a_1,u_2)(a_2,u_1)(a_1,a_2)}{(a_1a_2)^3(b_1b_2)^2}.
\]
Let $\mathscr{D}$ denote the sum on the right. To see that $\mathscr{D}$ converges, we take $\ep\leq\frac{1}{4}$ and express $\mathscr{D}$ as an Euler product $\mathscr{D} = \prod_p \mathscr{D}_p$ with
\[
\mathscr{D}_p = \sumfour_{a_1,b_1,a_2,b_2\geq 0} p^{y(a_1,b_1,a_2,b_2)},
\]
and
\[
\begin{aligned}
	y(a_1,b_1,a_2,b_2) &= \ep(a_1+b_1+a_2+b_2) + 2\min(a_1+b_1,a_2+b_2)+\min(a_1,u_2)+\min(a_2,u_1) \\
	&\hspace{.25in}+ \min(a_1,a_2)-3(a_1+a_2) - 2(b_1+b_2).
\end{aligned}
\]
Here we have written $u_i=a_i+b_i-\min(a_1+b_1,a_2+b_2)$. It suffices to show that for $a_i,b_i$ not all 0, we have $y(a_1,b_1,a_2,b_2) \leq -\frac{3}{2}$, say. We have trivially that
\[
\begin{aligned}
	y(a_1,b_1,a_2,b_2) &\leq \fract{1}{4}(a_1+b_1+a_2+b_2) + (a_1+b_1+a_2+b_2)+a_1+a_2 \\
	&\hspace{.25in}+ \fract{1}{2}(a_1+a_2)-3(a_1+a_2) - 2(b_1+b_2) \\
	&= - \fract{1}{4}(a_1+a_2) - \fract{3}{4}(b_1+b_2).
\end{aligned}
\]
Thus we may assume that $a_1+a_2\leq 5$ and $b_1+b_2\leq 1$. This leaves only a few cases to check, and one may verify by direct computation that we indeed have $y(a_1,b_1,a_2,b_2) \leq -\frac{3}{2}$ unless all the $a_i,b_i$ are 0. Therefore the sum converges, and we have
\[
\colr(q) \ll (\log q)^9.
\]

\section{Proof of Proposition \ref{prop:EEstimates}}\label{sec:EEstimates}

The analysis of the other 8 terms coming from Voronoi summation adheres closely to the analysis in Section 8 of \citep{CL}. Recall that by the decay of $U$, we may assume $a_i^3b_i^2N_i \ll q^{3/2+\ep}$. For $j=2,\ldots, 8$, let
\[
E_j(\bolda,\boldb,\boldN) = \sum_{c \leq C} \frac{1}{c}\sumstar_{x(c)} \colt_j(c,x).
\]
Changing variables in the sum over $c$ as in (\ref{eq:GCDs}) -- (\ref{eq:CSum1}), we have
\[
E_j(\bolda,\boldb,\boldN) = \sum_{\delta \leq C} \frac{1}{\delta} \sum_{c \leq C/\delta} \frac{1}{c} \sumstar_{\substack{x(c\delta )\\ (u_2 x-u_1,c\delta)=\delta}} \colt_j(c\delta ,x),
\]
where $u_1,u_2$ are as in (\ref{eq:GCDs}).

\subsection{The Sums $\colt_2,\ldots,\colt_5$}
Since each of these sums has the same form and behavior, we treat only $\colt_2$. The residue in the definition of $\colt_2$ gives
\[
\begin{aligned}
	&\sumstar_{\substack{x(c\delta )\\ (u_2 x-u_1,c\delta)=\delta}} \colt_j(c\delta ,x) \\
	&\hspace{.25in}=  \frac{\pi^\frac{3}{2}}{\eta_2^3} \sum_{m\geq 1} A_3^+\pth{m,\frac{\lambda_2}{\eta_2}} \int_0^\infty F_1(y_1) \pth{D_{-1}(\eta_1) +D_{-2}(\eta_1) \log y_1 + \fract{1}{2} D_{-3}(\eta_1)(\log y_1)^2}  \\
	&\hspace{.75in}  \times \int_0^\infty F_2(y_2) U_3\fracp{\pi^3my_2}{\eta_2^3} J_{k-1}\pth{\frac{4\pi}{c\delta q}\sqrt{a_1a_2y_1y_2}} dy_2\ dy_1 \sumstar_{\substack{x(c\delta )\\ (u_2 x-u_1,c\delta)=\delta}} 1
\end{aligned}
\]
where as before,
\[
\begin{aligned}
	F_1(y) &= y^{-\frac{1}{2}} f\fracp{y}{N} e\pth{\frac{a_1^2b_1y}{c\delta qa_2b_2}} U\fracp{a_1^3b_1^2y}{q^\frac{3}{2}}, \\
	F_2(y) &= y^{-\frac{1}{2}} f\fracp{y}{M} e\pth{\frac{a_2^2b_2y}{c\delta qa_1b_1}} U\fracp{a_2^3b_2^2y}{q^\frac{3}{2}}.
\end{aligned}
\]
By Lemma \ref{lem:Dbound}, we have $D_{-i}(\eta_1) \ll q^\ep\eta_1^{-1}$, and from (8.9) of \citep{Ivic}, we have
\begin{equation}\label{eq:Abound}
	A_3^\pm\pth{m,\frac{\lambda}{\eta}} \ll (\eta m)^\ep \eta^\frac{3}{2}m^\frac{1}{4}.
\end{equation}
We analyze the term coming from $D_{-1}$, as the analysis of the other two terms is nearly identitcal. Thus
\[
E_2(\bolda,\boldb,\boldN) \ll q^\ep \sum_{\delta \leq C} \sum_{c \leq C/\delta} \frac{1}{\eta_1 \eta_2^\frac{3}{2}} \sum_{m \geq 1} m^{\frac{1}{4}+\ep} \abs{I(m)},
\]
where
\[
\begin{aligned}
	I(m) &= \int_0^\infty F_1(y_1)\int_0^\infty   F_2(y_2) U_3\fracp{\pi^3my_2}{\eta_2^3} J_{k-1}\pth{\frac{4\pi}{c\delta q}\sqrt{a_1a_2y_1y_2}} dy_2\ dy_1,\\
	&=\int_0^\infty F_1(y_1) I_1(m,y_1)dy_1,
\end{aligned}
\]
say, and we have bounded the sum over $x$ trivially by $c\delta$. To estimate $E_2$, we write
\[
E_2(\bolda,\boldb,\boldN) = H_1+H_2,
\]
where $H_1$ is the contribution to $E_2$ from $m \leq q^\ep \eta_2^3/M$, and $H_2$ is the rest.

\subsubsection{The Contribution of $H_1$}

For brevity, put $C_1 = q^{-1}\sqrt{a_1a_2NM}$. Using (3.17) of \citep{Ivic}, which is
\begin{equation}\label{eq:U3epsilon}
	U_3(x) \ll x^\ep,
\end{equation}
and (\ref{eq:BesselMin}), we have
\[
I(m) \ll q^\ep (NM)^\frac{1}{2} \min\pth{\fracp{C_1}{c\delta}^{k-1}, \fracp{C_1}{c\delta}^{-\frac{1}{2}}},
\]
and so
\[
\begin{aligned}
	H_1 &\ll q^\ep \sum_{\delta \leq C} \sum_{c \leq C/\delta} \frac{1}{\eta_1 \eta_2^\frac{3}{2}} \sum_{m \leq \eta_2^3q^\ep/M} m^\frac{1}{4} (NM)^\frac{1}{2}\min\pth{\fracp{C_1}{c\delta}^{k-1}, \fracp{C_1}{c\delta}^{-\frac{1}{2}}} \\ 
	&\ll q^\ep \frac{N^\frac{1}{2}}{M^\frac{3}{4}} \sum_{\delta \leq C} \sum_{c \leq C/\delta}  \frac{\eta_2^\frac{9}{4}}{\eta_1}\min\pth{\fracp{C_1}{c\delta}^{k-1}, \fracp{C_1}{c\delta}^{-\frac{1}{2}}} \\
	&\ll q^\ep \frac{N^\frac{1}{2}a_1u_1^\frac{9}{4}}{M^\frac{3}{4}u_2} \sum_{\delta \leq C} \sum_{c \leq C/\delta} c^\frac{5}{4} \min\pth{\fracp{C_1}{c\delta}^{k-1}, \fracp{C_1}{c\delta}^{-\frac{1}{2}}} \\
	&\ll q^\ep \frac{N^\frac{1}{2}a_1u_1^\frac{9}{4}}{M^\frac{3}{4}u_2} C_1^\frac{9}{4} \\
	&\ll (a_1^3b_1^2N)^\frac{13}{8} (a_2^3b_2^2M)^\frac{3}{8} q^{-\frac{9}{4}+\ep} \\
	&
	= q^{\frac{3}{4}+\ep}.
\end{aligned}
\]
Here we have used the estimates 
\[
\eta_2 \leq u_1c, \qquad \eta_1\geq \frac{u_2c}{a_1}.
\]
Summing the above estimate over $\bolda,\boldb,\boldN$ gives the desired result. All of the estimates that follow will be sufficient when summed over these variables, so we omit this sort of remark in what follows.
\subsubsection{The Contribution of $H_2$}\label{sec:H2}
To handle $H_2$, we use the following identity for $U_3$, given by (3.12) of \citep{Ivic}. For some suitable constants $c_j,d_j$, we have
\begin{equation}\label{eq:IvicU}
	U_3(\pi ^3 x) = \sum_{j=1}^K \frac{1}{x^\frac{j}{3}} \pth{c_j e\pth{3x^\frac{1}{3}} + d_je\pth{-3x^\frac{1}{3}}} + O\fracp{1}{x^{(K+1)/3}}.
\end{equation}
In the present case, we have $\frac{\pi^3 m y_2}{\eta_1^3} \gg q^\ep$. Thus
\begin{equation}\label{eq:I1withU3series}
	\begin{aligned}
		I_1(m,y_1) &= \sum_{j=1}^K \fracp{\eta_2}{M^\frac{1}{3}m^\frac{1}{3}}^j \int_0^\infty F_2(y_2)\fracp{M}{y_2}^\frac{j}{3} \pth{c_je\fracp{ 3m^\frac{1}{3}y_2^\frac{1}{3}}{\eta_2} + d_j e\pth{-\frac{3m^\frac{1}{3}y_2^\frac{1}{3}}{\eta_2}}} \\
		&\hspace{1in}\times J_{k-1}\pth{\frac{4\pi}{c\delta q}\sqrt{a_1a_2y_1y_2}} dy_2 + O\pth{q^{-2022}},
	\end{aligned}
\end{equation}
for $K$ sufficiently large in terms of $\ep$. This also gives the trivial bound 
\begin{equation}\label{eq:H2Trivial}
	I(m)\ll \frac{\eta_2}{m^\frac{1}{3}}N^\frac{1}{2}M^\frac{1}{6} \ll q^\ep (NM)^\frac{1}{2}.
\end{equation}
Let $C_2 = 8\pi(q\delta)^{-1}\sqrt{a_1a_2NM}$. We divide into two cases depending as $c \leq C_2$ and $c > C_2$.\\

\noindent \textbf{Case 1: $c > C_2$.} Using (\ref{eq:BesselSeries}), we can write $I_1$ as
\[
\begin{aligned}
	&I_1(m,y_1) = \sum_{j=1}^K \fracp{\eta_2}{M^\frac{1}{3}m^\frac{1}{3}}^j \sum_{\ell=0}^\infty  \frac{(-1)^\ell}{\ell! (\ell+k-1)!}  \\
	&\hspace{1in}\times\int_0^\infty \colf_j(y_1,y_2,\ell)\pth{c_je\fracp{ 3m^\frac{1}{3}y_2^\frac{1}{3}}{\eta_2} + d_j e\pth{-\frac{3m^\frac{1}{3}y_2^\frac{1}{3}}{\eta_2}}} e\pth{\frac{a_2^2b_2y_2}{c\delta qa_1b_1}} dy_2, 
\end{aligned}
\]
where
\[
\colf_j(y_1,y_2,\ell) = y_2^{-\frac{1}{2}} \fracp{y_2}{M}^{-\frac{j}{3}} f\fracp{y_2}{M} \pth{\frac{2}{c\delta q}\sqrt{a_1a_2y_1y_2}}^{2\ell+k-1} U\fracp{a_2^3b_2^2y_2}{q^\frac{3}{2}}.
\]
We now analyze the integrals
\[
\int_0^\infty \colf_j(y_1,y_2,\ell) e(\omega_\pm(m,y_2)) dy_2,
\]
where
\[
\omega_\pm(m,y_2) = \pm \frac{3m^\frac{1}{3}}{\eta_2}y_2^\frac{1}{3} + By_2, \qquad B = \frac{a_2^2b_2}{c\delta qa_1b_1}.
\]
We have
\[
\omega_\pm'(m,y_2) = \pm \frac{m^\frac{1}{3}}{y_2^\frac{2}{3}\eta_2} + B
\]
If $m > 64 (B\eta_2)^3 M^2$ or $m < \frac{1}{64}(B\eta_2)^3 M^2$, we have $\omega_\pm'(m,y_2) \gg \frac{m^{\frac{1}{3}}}{y_2^\frac{2}{3}\eta_2} \gg \frac{q^\ep}{M}$. Thus the contribution of these terms is negligible by integrating by parts many times. Thus we need only consider those $m$ for which $m \asymp (B\eta_2)^3 M^2$. But since
\[
(B\eta_2)^3M^2 \ll \frac{q^\ep}{\delta^3},
\]
there are no terms of this form unless $M \gg \frac{q^\frac{3}{2}}{(a_2^2b_2)^\frac{3}{2}}$ and $\delta \ll q^\ep$. Using the trivial bound (\ref{eq:H2Trivial}) we see that the contribution to $H_2$ of these terms is bounded by
\[
\begin{aligned}
	q^\ep N^\frac{1}{2}M^\frac{1}{6}\sum_{\delta \ll q^\ep} \sum_{c> C_2} \frac{1}{\eta_1 \eta_2^\frac{1}{2}} \sum_{m\ll q^\ep} m^{-\frac{1}{12}+\ep} &\ll q^\ep N^\frac{1}{2}M^\frac{1}{6}a_1^{\frac{5}{4}}b_1^{\frac{1}{4}}a_2^{\frac{1}{4}}b_2^{-\frac{1}{4}}\sum_{\delta \ll q^\ep} \sum_{c> C_2} c^{-\frac{3}{2}} \\
	&\ll q^\ep N^\frac{1}{2}M^\frac{1}{6}a_1^{\frac{5}{4}}b_1^{\frac{1}{4}}a_2^{\frac{1}{4}}b_2^{-\frac{1}{4}} \fracp{q}{\sqrt{NMa_1a_2}}^{\frac{1}{2}} \\
	&\ll q^{\frac{1}{2}+\ep} a_1b_1^\frac{1}{4}N^\frac{1}{4}b_2^{-\frac{1}{4}}M^{-\frac{1}{12}}\\
	&\ll q^{\frac{1}{2}+\ep} b_1^\frac{1}{4}N^\frac{1}{4}b_2^{-\frac{1}{4}} \fracp{(a_2^2b_2)^\frac{3}{2}}{q^\frac{3}{2}}^\frac{1}{12} \\
	&\ll q^{\frac{3}{8}+\ep} (a_1a_2)^\frac{1}{4} \pth{a_1^3b_1^2N}^\frac{1}{4} \\
	&\ll (a_1a_2)^\frac{1}{4} q^{\frac{3}{4}+\ep}.
\end{aligned}
\]
Here we have used the estimate
\[
\eta_1 \eta_2^\frac{1}{2} \geq \frac{u_2c}{a_1}\fracp{u_1c}{a_2}^{\frac{1}{2}} = c^{\frac{3}{2}} \frac{b_1^\frac{1}{2} a_2^\frac{1}{2}b_2}{a_1^\frac{1}{2}(a_1b_1,a_2b_2)^\frac{3}{2}} \geq c^{\frac{3}{2}}\frac{b_1^\frac{1}{2} a_2^\frac{1}{2}b_2}{a_1^\frac{1}{2}(a_1b_1a_2b_2)^\frac{3}{4}} = c^{\frac{3}{2}} a_1^{-\frac{5}{4}}b_1^{-\frac{1}{4}}a_2^{-\frac{1}{4}}b_2^{\frac{1}{4}}.
\]

\noindent \textbf{Case 2: $c \leq C_2$.} We return to (\ref{eq:I1withU3series}), but instead use (\ref{eq:BesselW}) in place of (\ref{eq:BesselSeries}) to write $I_1$ as
\[
\begin{aligned}
	I_1(m,y_1) &= \sum_{j=1}^K \fracp{\eta_2}{M^\frac{1}{3}m^\frac{1}{3}}^j \int_0^\infty F_2(y_2)\fracp{M}{y_2}^\frac{j}{3} \pth{c_je\fracp{ 3m^\frac{1}{3}y_2^\frac{1}{3}}{\eta_2} + d_j e\pth{-\frac{3m^\frac{1}{3}y_2^\frac{1}{3}}{\eta_2}}} \\
	&\times \pth{\frac{c\delta q}{2\sqrt{a_1a_2y_1y_2}}}^\frac{1}{2} \pth{2\Re W\pth{\frac{4\pi}{c\delta q}\sqrt{a_1a_2y_1y_2}}e\pth{\frac{2}{c\delta q}\sqrt{a_1a_2y_1y_2} - \frac{k}{4}+\frac{1}{8}}} dy_2 \\
	&+ O\pth{q^{-2022}}.
\end{aligned}
\]
Note that this gives the trivial bound
\begin{equation}\label{eq:H2smallcTrivialI}
	I(m) \ll \frac{\eta_2}{m^\frac{1}{3}} \fracp{c\delta q}{\sqrt{a_1a_2}}^\frac{1}{2} N^\frac{1}{4} M^{-\frac{1}{12}}.
\end{equation}
Define
\[
\colh_j(y_1,y_2) = y_1^{-\frac{1}{4}} y_2^{-\frac{3}{4}} \fracp{M}{y_2}^{\frac{j}{3}}  f\fracp{y_2}{M} W_1\pth{\frac{4\pi}{c\delta q}\sqrt{a_1a_2y_1y_2}} U\fracp{a_2^3b_2^2y_2}{q^\frac{3}{2}},
\]
where $W_1$ is either $W$ or $\bar{W}$. For some absolute constants $b_j$, we find that $I_1$ is (up to a negligible error term) a sum of expressions of the form
\[
I_1(m,y_1) = \fracp{\sqrt{c\delta q}}{(a_1a_2)^\frac{1}{4}} \sum_{j=1}^K b_j \fracp{\eta_2}{M^\frac{1}{3}m^\frac{1}{3}}^j \int_0^\infty \colh_j(y_1,y_2)e(\omega(y_1,y_2))\ dy_2,
\]
where
\[
\omega(y_1,y_2) = \pm \frac{3m^\frac{1}{3}}{\eta_2}y_2^\frac{1}{3} \pm 2A y_2^\frac{1}{2} + B y_2, \qquad A = \frac{(a_1a_2y_1)^\frac{1}{2}}{c\delta q},\ \ B = \frac{a_2^2b_2}{c\delta q a_1b_1}.
\]
Differentiating with respect to $y_2$ gives
\[
\omega'(y_1,y_2) = \pm \frac{m^\frac{1}{3}}{\eta_2}y_2^{-\frac{2}{3}} \pm A y_2^{-\frac{1}{2}} + By_2.
\]
We divide into several cases. \\

\noindent \textit{Case 1.1:} $a_1^\frac{3}{2}b_1N^\frac{1}{2} \geq 4a_2^\frac{3}{2}b_2M^\frac{1}{2}$. Then it is easily checked that
\[
\frac{1}{2} \frac{A}{y_2^\frac{1}{2}} \leq \abs{\pm A y_2^{-\frac{1}{2}} + By_2} \leq 2 \frac{A}{y_2^\frac{1}{2}}.
\]
If $m \geq 64 (A\eta_2)^3 M^\frac{1}{2}$ or $m \leq \frac{1}{64} (A\eta_2)^3 M^\frac{1}{2}$, then $\abs{\omega'(y_1,y_2)} \gg \frac{m^\frac{1}{3}}{y_2^\frac{2}{3}\eta_2} \gg \frac{q^\ep}{M}$, since $n \gg \frac{\eta_2^3q^\ep}{M}$. Thus we may integrate by parts many times to see that the contribution of these terms is negligible. For the terms with $\frac{1}{64} (A\eta_2)^3 M^\frac{1}{2} \leq m \leq 64 (A\eta_2)^3 M^\frac{1}{2}$, note that
\[
(A\eta_2)^3 M^\frac{1}{2} \ll \frac{(a_1^3b_1^2N)^{\frac{3}{2}}(a_2^3M)^\frac{1}{2}}{q^3\delta^3} \ll \frac{q^\ep}{\delta^3}.
\]
Moreover, the left side is only $\gg 1$ if $M \gg q^\frac{3}{2}/a_2^3$ and $\delta \ll q^\ep$. By (\ref{eq:H2smallcTrivialI}), the contribution from these terms is bounded by
\[
\begin{aligned}
	q^{\frac{1}{2}+\ep} \frac{(NM)^\frac{1}{4}}{M^{\frac{1}{3}}(a_1a_2)^\frac{1}{4}} &\sum_{\delta \leq q^\ep} \delta^\frac{1}{2}\sum_{c \leq C_2} \frac{c^\frac{1}{2}}{\eta_1 \eta_2^\frac{1}{2}} \sum_{m \ll q^\ep} m^{-\frac{1}{12}+\ep} \\
	&\ll q^{\frac{1}{2}+\ep} \frac{(NM)^\frac{1}{4}}{M^{\frac{1}{3}}(a_1a_2)^\frac{1}{4}} a_1a_2^\frac{1}{2} \\
	&= a_2^\frac{1}{2} q^{\frac{1}{2}+\ep}  (a_1^3N)^\frac{1}{4} (a_2^3M)^\frac{1}{4}  (a_2^3M)^{-\frac{1}{3}} \\
	&\ll a_2^\frac{1}{2}q^{\frac{3}{4}+\ep},
\end{aligned}
\]
where we have used the trivial bound $\eta_i\geq \frac{c}{a_i}$.\\

\noindent \textit{Case 1.2:} $a_2^\frac{3}{2}b_2M^\frac{1}{2} \geq 4a_1^\frac{3}{2}b_1N^\frac{1}{2}$. Then as before, one checks that $\frac{1}{2} B \leq \abs{\pm Ay_2^{-\frac{1}{2}} + B} \leq \frac{3}{2}B$. By the same arguments as in the previous case, the range of $m$ that should be considered is of size $(B\eta_2)^3M$, and the contribution to $H_2$ of this range is bounded by $a_2^\frac{1}{2} q^{\frac{3}{4}+\ep}$. \\

\noindent \textit{Case 1.3:}  $\frac{1}{4}  a_2^\frac{3}{2}b_2M^\frac{1}{2} < a_1^\frac{3}{2}b_1N^\frac{1}{2} <4  a_2^\frac{3}{2}b_2M^\frac{1}{2}$. In this case $Ay_2^{-\frac{1}{2}} \asymp B$, and so the range of $m$ that should be considered is of size $(A\eta_2)^3M^\frac{1}{2}$ by the same arguments above, and the contribution from this range is also bounded by $a_2^\frac{1}{2} q^{\frac{3}{4}+\ep}$. 

\subsection{The Sums $\colt_6,\ldots,\colt_9$}
Each of the four sums $\colt_6,\ldots,\colt_9$ has essentially the same form and behavior, so we deal only with $\colt_6$. The treatment these sums is very similar to that of $\colt_2,\ldots, \colt_5$, so we shall be somewhat brief. Recall that
\[
\begin{aligned}
	E_6(\bolda,\boldb,\boldN) &= \sum_{\delta \leq C} \frac{1}{\delta} \sum_{c \leq C/\delta} \frac{1}{c} \sumstar_{\substack{x(c\delta )\\ (u_2 x-u_1,c\delta)=\delta}} \frac{\pi^3}{\eta_1^3\eta_2^3} \sum_{n\geq 1} \sum_{m\geq 1} A_3^+\pth{n,\frac{\lambda_1}{\eta_1}} A_3^+\pth{m,\frac{\lambda_2}{\eta_2}}  \\
	&\times \int_{0}^{\infty} \int_{0}^{\infty} F_1(y_1) F_2(y_2) U_3\fracp{\pi^3ny_1}{\eta_1^3} U_3\fracp{\pi^3my_2}{\eta_2^3} J_{k-1}\pth{\frac{4\pi}{c\delta q} \sqrt{a_1y_1a_2y_2}} dy_1\ dy_2,
\end{aligned}
\]
where 
\[
\eta_1 = \frac{u_2 c}{(a_1,u_2 c)}, \qquad \eta_2 = \frac{u_1 c}{(a_2,u_1 c)}.
\]
As before, we use (\ref{eq:Abound}) and estimate the sum over $x$ trivially by $c\delta$ to see that
\[
E_6(\bolda,\boldb,\boldN) \ll \sum_{\delta \leq C} \sum_{c \leq C/\delta}  \frac{1}{(\eta_1\eta_2)^\frac{3}{2}} \sum_{n\geq 1} \sum_{m\geq 1} (nm)^{\frac{1}{4}+\ep} \abs{I(n,m)},
\]
where
\[
I(n,m) = \int_{0}^{\infty} \int_{0}^{\infty} F_1(y_1) F_2(y_2) U_3\fracp{\pi^3ny_1}{\eta_1^3} U_3\fracp{\pi^3my_2}{\eta_2^3} J_{k-1}\pth{\frac{4\pi}{c\delta q} \sqrt{a_1y_1a_2y_2}} dy_1\ dy_2.
\]
Recall that it suffices to show that $E_6 \ll (a_1a_2)^\frac{1}{2}q^{\frac{3}{4}+\ep}$. Following our previous analysis, we write
\[
E_6 = \sum_{i=1}^4 E_{6,i},
\]
where $E_{6,i}$ is the contribution to $E_6$ from case $i$ below.
\begin{enumerate}[label=(\arabic*)]
	\item $n \ll \frac{\eta_1^3q^\ep}{N}$ and $m \ll \frac{\eta_2^3q^\ep}{M}$;
	\item $n \gg \frac{\eta_1^3q^\ep}{N}$ and $m \ll \frac{\eta_2^3q^\ep}{M}$;
	\item $n \ll \frac{\eta_1^3q^\ep}{N}$ and $m \gg \frac{\eta_2^3q^\ep}{M}$;
	\item $n \gg \frac{\eta_1^3q^\ep}{N}$ and $m \gg \frac{\eta_2^3q^\ep}{M}$.
\end{enumerate}
By symmetry, the treatment of cases (2) and (3) is the same, so we treat only
the second case.

\subsubsection{The Contribution of $E_{6,1}$}

For this case, we use the estimate
\[
U_3\fracp{\pi^3n_iy_i}{\eta_i^3} \ll q^\ep
\]
along with (\ref{eq:BesselMin}) to see that
\[
I(n,m) \ll q^\ep (NM)^\frac{1}{2}\min\pth{\fracp{\sqrt{a_1a_2NM}}{c\delta q}^{-\frac{1}{2}}, \fracp{\sqrt{a_1a_2NM}}{c\delta q}^{k-1}},
\]
and so
\[
\begin{aligned}
	E_{6,1}(\bolda,\boldb,\boldN) &\ll q^\ep (NM)^\frac{1}{2}\sum_{\delta \leq C} \sum_{c \leq C/\delta}  \frac{1}{(\eta_1\eta_2)^\frac{3}{2}} \min\pth{\fracp{\sqrt{a_1a_2NM}}{c\delta q}^{-\frac{1}{2}}, \fracp{\sqrt{a_1a_2NM}}{c\delta q}^{k-1}} \\
	&\times \sum_{n\ll \frac{\eta_1^3q^\ep}{N}} \sum_{m\ll \frac{\eta_2^3q^\ep}{M}} (nm)^{\frac{1}{4}+\ep} \\
	&\ll q^\ep \frac{(u_1u_2)^\frac{9}{4}}{(NM)^{-\frac{3}{4}}} \sum_{\delta \leq C} \fracp{\sqrt{a_1a_2NM}}{\delta q}^\frac{13}{4} \\
	&\ll q^{-\frac{13}{4}+\ep} (a_1a_2)^{\frac{31}{8}} (b_1b_2)^\frac{9}{4} (NM)^\frac{7}{8} \\
	&\ll q^{-\frac{13}{4}+\ep} (a_1a_2)^\frac{1}{2} (a_1^3b_1^2N)^\frac{9}{8}(a_2^3b_2^2M)^\frac{9}{8} \\
	&\ll (a_1a_2)^\frac{1}{2}q^{\frac{1}{8}+\ep}.
\end{aligned}
\]

\subsubsection{The Contribution of $E_{6,2}$}

We write 
\[
I(n,m) = \int_{0}^{\infty} F_2(y_2) U_3\fracp{\pi^3my_2}{\eta_2^3} I_1(n,y_2) dy_1\ dy_2.
\]
The integration in $y_2$ can be bounded trivially and the sum over $m$ can be treated as in the previous subsection. The integral $I_1(n,y_2)$ can be handled in the same way as cases 1 and 2 in the Section \ref{sec:H2}, and we obtain $E_{6,2} \ll a_1^\frac{1}{2}q^{\frac{1}{2}+\ep}$.

\subsubsection{The Contribution of $E_{6,4}$}

As in Section \ref{sec:H2}, we let $C_2 = 8\pi (q\delta)^{-1} \sqrt{a_1a_2NM}$ and divide into two cases depending as $c \leq C_2$ and $c > C_2$. \\

\noindent \textbf{Case 1: $c > C_2$}. We again use (\ref{eq:BesselSeries}) and (\ref{eq:IvicU}) and consider integrals of the form
\[
\int_{0}^{\infty} \int_{0}^{\infty} \colh(y_1,y_2) e\pth{\frac{a_1^2b_1y_1}{c\delta q a_2b_2} +\frac{a_2^2b_2y_2}{c\delta q a_1b_1} \pm \frac{3n^\frac{1}{3}y_1^\frac{1}{3}}{\eta_1}  \pm \frac{3m^\frac{1}{3}y_2^\frac{1}{3}}{\eta_2}} dy_1\ dy_2,
\]
where $\frac{\partial^j\partial^k \colh(y_1,y_2)}{\partial y_1^j \partial y_2^k} \colh(y_1,y_2) \ll N^{-j} M^{-k}$, $\colh(y_1,y_2) \ll 1$, and is supported on $[N,2N] \times [M,2M]$. Thus the range of integration is $O(NM)$.

By the same arguments as in Case 1 of Section \ref{sec:H2}, it suffices to consider the case when $c_1 (B_1\eta_1)^3N^2 \ll n \ll c_2 (B_1\eta_1)^3N^2$ and $c_1 (B_2\eta_2)^3M^2 \ll m \ll c_2 (B_2\eta_2)^3M^2$, where $c_1,c_2$ are some constants,
\[
B_1 = \frac{a_1^2b_1}{c\delta q a_2b_2} \mand B_2 = \frac{a_2^2b_2}{c\delta q a_1b_1}.
\]
The terms outside these ranges give negligible contribution from integration by parts many times. As before, we have
\[
(B_1\eta_1)^3 N^2 \ll \frac{(a_1^2b_1)^3N^2}{\delta^3 q^3} \ll \frac{q^\ep}{\delta^3}, \quad (B_2\eta_2)^3 M^2 \ll \frac{(a_2^2b_2)^3M^2}{\delta^3 q^3} \ll \frac{q^\ep}{\delta^3}.
\]
There are no terms of this form unless 
\[
N \gg \frac{q^{\frac{3}{2}}}{(a_1^2b_1)^{\frac{3}{2}}}, \qquad M \gg \frac{q^{\frac{3}{2}}}{(a_2^2b_2)^{\frac{3}{2}}}, \qquad \delta \ll q^\ep.
\]
The contribution to $E_{6,4}$ of the terms with $c > C_2$ is bounded by
\begin{equation}\label{eq:E64largecterms}
	\frac{q^\ep(NM)^{\frac{1}{2}}}{(u_1u_2)^\frac{3}{2}} \sum_{\delta \ll q^\ep} \sum_{c > C_2} \frac{((a_1,u_2c)(a_2,u_1c))^{\frac{3}{2}}}{c^3}.
\end{equation}
To estimate the sum over $c$, let 
\[
\begin{aligned}
	g_1=(a_1,u_2), \quad &&a_1=\lambda_1 g_1, \quad &&u_2 = \gamma_1 g_1, \\
	g_2=(a_2,u_1), \quad &&a_2=\lambda_2 g_2, \quad &&u_1 = \gamma_2 g_2,
\end{aligned}
\]
and
\[
d = (\lambda_1,\lambda_2), \quad \lambda_1 = \alpha_1d, \quad \lambda_2 = \alpha_2 d,
\]
where $(\lambda_1,\gamma_1) = (\lambda_2,\gamma_2) = (\alpha_1,\alpha_2)= 1$. Then the sum over $c$ is
\[
\begin{aligned}
	&(g_1g_2)^{\frac{3}{2}} \sum_{c > C_2} \frac{((\lambda_1,c)(\lambda_2,c))^{\frac{3}{2}}}{c^3} = (g_1g_2)^{\frac{3}{2}} \sum_{\ell\mid d} \sum_{\substack{c > C_2\\ (c,d)=\ell}} \frac{((\alpha_1d,c)(\alpha_2 d,c))^{\frac{3}{2}}}{c^3} \\
	&\hspace{.25in}=(g_1g_2)^{\frac{3}{2}} \sum_{\ell\mid d} \sum_{\substack{c > \frac{C_2}{\ell}\\ (c,\frac{d}{\ell})=1}} \frac{(\alpha_1\alpha_2 ,c)^{\frac{3}{2}}}{c^3} = (g_1g_2)^{\frac{3}{2}} \sum_{\ell\mid d} \sum_{k\mid \alpha_1\alpha_2} k^{\frac{3}{2}} \sum_{\substack{c > \frac{C_2}{\ell}\\ (c,\frac{d}{\ell})=1\\ (\alpha_1\alpha_2,c)=k}} \frac{1}{c^3} \\
	&\hspace{.25in}= (g_1g_2)^{\frac{3}{2}} \sum_{\ell\mid d} \sum_{k\mid \alpha_1\alpha_2} k^{-\frac{3}{2}} \sum_{\substack{c > \frac{C_2}{\ell k}\\ (c,\frac{d}{\ell})=1\\ (\frac{\alpha_1\alpha_2}{k},c)=1}} \frac{1}{c^3} \ll \frac{(g_1g_2)^{\frac{3}{2}}}{C_2^2} \sum_{\ell\mid d} \ell^2 \sum_{k\mid \alpha_1\alpha_2} k^{\frac{1}{2}} \\
	&\ll \delta^2 q^{2+\ep}\frac{(g_1g_2)^{\frac{3}{2}}}{a_1a_2NM} d^2 (\alpha_1\alpha_2)^\frac{1}{2} \ll \delta^2 q^{2+\ep}\frac{(a_1,u_2)(a_2,u_1)(a_1,a_2)}{NM(a_1a_2)^\frac{1}{2}}.
\end{aligned}
\]
Thus the total contribution is bounded by
\[
\begin{aligned}
	q^{2+\ep}\frac{(a_1,u_2)(a_2,u_1)(a_1,a_2)}{(u_1u_2)^\frac{3}{2}(a_1a_2)^\frac{1}{2}(NM)^\frac{1}{2}} &\ll q^{\frac{1}{2}+\ep}\frac{(a_1,u_2)(a_2,u_1)(a_1,a_2)}{(u_1u_2)^\frac{3}{2}(a_1a_2)^\frac{1}{2}} \pth{a_1^\frac{3}{2}b_1^\frac{3}{4}a_2^\frac{3}{2}b_2^\frac{3}{4}} \\
	&\ll q^{\frac{1}{2}+\ep} \frac{(a_1,a_2)}{(u_1u_2a_1a_2)^\frac{1}{2}}\pth{a_1^\frac{3}{2}b_1^\frac{3}{4}a_2^\frac{3}{2}b_2^\frac{3}{4}} = q^{\frac{1}{2}+\ep} (a_1,a_2)(a_1b_1,a_2b_2) (b_1b_2)^\frac{1}{4} \\
	&\ll q^{\frac{7}{8}+\ep} (a_1,a_2)(a_1b_1,a_2b_2).
\end{aligned}
\]
Summing this over $a_1,a_2,b_1,b_2$ produces several factors of $\log q$, and thus the contribution from these terms is sufficiently small.\\

\noindent \textbf{Case 2:} $c\leq C_2$. We proceed as in the last case, except that we use (\ref{eq:BesselW}) in place of (\ref{eq:BesselSeries}). The integrals we consider have the form
\[
\int_{0}^{\infty} \int_{0}^{\infty} G(y_1,y_2) e\pth{\phi(m,n,y_1,y_2)} dy_1\ dy_2,
\]
where
\[
\phi(m,n,y_1,y_2) = \frac{a_1^2b_1y_1}{c\delta q a_2b_2} +\frac{a_2^2b_2y_2}{c\delta q a_1b_1} \pm \frac{3n^\frac{1}{3}y_1^\frac{1}{3}}{\eta_1}  \pm \frac{3m^\frac{1}{3}y_2^\frac{1}{3}}{\eta_2} \pm \frac{2\sqrt{a_1a_2y_1y_2}}{c\delta q}
\]
$\frac{\partial^j\partial^k G(y_1,y_2)}{\partial y_1^j \partial y_2^k} G(y_1,y_2) \ll N^{-j} M^{-k}$, $G(y_1,y_2) \ll 1$. As before, the range of integration is $O(NM)$. Then
\[
\begin{aligned}
	\frac{\partial \phi(m,n,y_1,y_2)}{\partial y_1} &= B_1 \pm \frac{n^{\frac{1}{3}}}{y_1^\frac{2}{3}\eta_1} \pm \frac{A_1}{y_1^\frac{1}{2}}, \\
	\frac{\partial \phi(m,n,y_1,y_2)}{\partial y_2} &= B_2 \pm \frac{m^{\frac{1}{3}}}{y_1^\frac{2}{3}\eta_2} \pm \frac{A_2}{y_2^\frac{1}{2}},\\
\end{aligned}
\]
where $B_1,B_2$ are as above and $A_1 = \frac{\sqrt{a_1a_2y_2}}{c\delta q}$ and $A_2 =\frac{\sqrt{a_1a_2y_1}}{c\delta q}$. We now follow closely the analysis for Case 2 of Section \ref{sec:H2}, dividing into several subcases.

\noindent \textit{Case 2.1}: $a_2^\frac{3}{2}b_2 M^\frac{1}{2} \geq 4a_1^\frac{3}{2} b_1 N^\frac{1}{2}$. For this case, we have $\abs{\frac{A_1}{y_1^\frac{1}{2}} \pm B_1} \asymp \frac{A_1}{y_1}$ and $\abs{\frac{A_2}{y_2^\frac{1}{2}} \pm B_2} \asymp B_2$. By similar arguments to Case 2 of Section \ref{sec:H2}, we consider the ranges $n\asymp (A_1\eta_1)^3 N^\frac{1}{2}$ and $m\asymp (B_2\eta_2)^3 M^2$ and note that
\[
(A_1\eta_1)^3 N^\frac{1}{2} \ll \fracp{\sqrt{a_1}}{\delta q^\frac{1}{2}}^3 N^\frac{1}{2} \ll \frac{q^\ep}{\delta^3}, \qquad (B_2 \eta_2)^3 M^2 \ll \fracp{(a_2^2b_2)^3 M^2}{\delta^3 q^3} \ll \frac{q^\ep}{\delta^3},
\]
and thus there are no terms of this form unless $N \gg \frac{q^\frac{3}{2}}{a_1^3}$, $M \gg \frac{q^{\frac{3}{2}}}{(a_2^2b_2)^\frac{3}{2}}$ and $\delta \ll q^\ep$. The contribution from these terms to $E_{6,4}$ is $O(a_1^\frac{1}{2} a_2^\frac{1}{2} q^{\frac{1}{2}+\ep})$.\\

\noindent \textit{Case 2.2:} $a_1^\frac{3}{2}b_1 N^\frac{1}{2} \geq 4a_2^\frac{3}{2} b_2 N^2$. The calculation as in Case 2.1 gives and error of  $O(a_1^\frac{1}{2} a_2^\frac{1}{2} q^{\frac{1}{2}+\ep})$.\\

\noindent \textit{Case 2.3:} $\frac{1}{4} a_1^\frac{3}{2} b_1 N^\frac{1}{2} < a_2^\frac{3}{2} b_2 M^\frac{1}{2} < 4 a_1^\frac{3}{2} b_1 N^\frac{1}{2}$. For this case, we have $\frac{A_1}{y_1^\frac{1}{2}} \asymp B_1$ and $\frac{A}{y_2^\frac{1}{2}} \asymp B_2$. By similar arguments to Cases 1.2 and 1.3, we need only consider the ranges $n \asymp (A_1\eta_1)^3 N^\frac{1}{2}$ and $m \asymp (A_1\eta_1)^3 N^\frac{1}{2}$. The contribution from these terms is also $O\pth{a_1^\frac{1}{2} a_2^\frac{1}{2} q^{\frac{1}{2}+\ep}}$.
\bibliography{references}
\end{document}